\documentclass[psamsfonts]{amsart}
\usepackage{mymacros}

\title{Deformation Quantization for Supermanifolds via Gelfand-Kazhdan Descent}
\date{}
\author{Araminta Amabel}

\begin{document}

\begin{abstract}
We construct a canonical deformation quantization for symplectic supermanifolds.  
This gives a novel proof of the super-analogue of Fedosov quantization.
Our proof uses the formalism of Gelfand-Kazhdan descent, 
whose foundations we establish in the super-symplectic setting. 
\end{abstract}
 
\maketitle
\tableofcontents

\section{Introduction}

Given a symplectic manifold $(M,\omegans)$, 
it is a classic question to ask whether there exists a deformation of the algebra of functions $\mathcal{O}_M$ compatible with the symplectic form $\omegans$. 
The space of such deformations was first described independently by De Wilde-Lecomte \cite{DWL} and Fedosov \cite{Fedosov}. 
This result was extended by Kontsevich to apply to all Poisson manifolds, \cite{Kontsevich}.

Here, we give a new proof of the super-analogue of Fedosov's quantization result, 
showing that for symplectic supermanifolds there exists a deformation quantization. 
A symplectic supermanifold is a supermanifold $\bb{M}$ together with an even, closed, nondegenerate 2-form $\omega$ on $\bb{M}$, 
(Definition \ref{def-sympsuper}). 
In particular, we work with \emph{even} symplectic supermanifolds. 

Fedosov's quantization of a (non-super) symplectic manifold $(M,\omegans)$ requires the data of a symplectic connection on $M$. 
In our formulation, for a symplectic supermanifold $(\bb{M},\omega)$, 
the connection data is replaced with an \emph{$\hbar$-formal exponential} 
(Definitions \ref{def-FormalExponential} and \ref{def-hExp}).

\begin{thm}\label{thm-1}
Let $(\bb{M},\omega,\sigma)$ be a symplectic supermanifold with an $\hbar$-formal exponential $\sigma$. 
Then there exists a canonical deformation quantization $\cal{A}_\sigma(\bb{M})$ of the Poisson superalgebra $\mathcal{O}_\bb{M}$. 
\end{thm}

\noindent This is Theorem \ref{thm-SuperFedosov} below. 
For $(M,\omega)$ a non-super symplectic manifold, 
a symplectic connection on $M$ determines an $\hbar$-formal exponential. 
In this case, our theorem recovers Fedosov's canonical deformation. 
We discuss the relationship between the deformation quantization in Theorem \ref{thm-1} and the space of all deformation quantizations in Remark \ref{rmk-BKours}.

A \emph{deformation quantization} of a Poisson $\bb{k}$-algebra $(A,\{-,-\})$ 
(such as $\bb{k}$-valued functions on a symplectic manifold) 
is an associative $\bb{k}[[\hbar]]$-algebra $(A_\hbar,\star)$ with a $\bb{k}[[\hbar]]$-module isomorphism $A_\hbar\simeq A[[\hbar]]$ such that 
\begin{itemize}
\item for all $f,g\in A_\hbar$, we have 
\[f\star g=fg+\hbar B_1(f,g)+\hbar^2 B_2(f,g)+\cdots\]
 for bidifferential operators $B_i(-,-)$, and 

\item if $f,g\in A$, then $\frac{1}{\hbar}[f,g]=\{f,g\}\mod{\hbar}$. 
\end{itemize}

The simplest example of a symplectic manifold is the cotangent bundle $T^*\bb{R}^n$. 
The cotangent bundle has a canonical quantization given by the Weyl algebra (Definition \ref{def-WeylAlgebra}). 
By Darboux's Lemma, symplectic manifolds all locally look like $T^*\bb{R}^n$ for some $n$. 
Production of a deformation quantization on a general symplectic manifold usually proceeds by trying to globalize from the Weyl algebra on a Darboux chart. 
For example, given a manifold $X$, 
the canonical deformation quantization of $T^*X$ given by the Rees algebra of differential operators on $X$ can be produced locally using the Weyl algebra.

There is a similar story for symplectic supermanifolds, 
but the local structure in the odd direction has additional freedom. 
Locally, a symplectic supermanifold is specified by what we call it's \emph{type}: 
a triple of numbers $(2n\vert a,b)$ where 
there are $2n$ even dimensions, 
$a+b$ odd dimensions, 
and the symplectic structure in the odd direction comes from a quadratic form $Q$ of signature $(a,b)$. 
The canonical local quantization is then a tensor product of Weyl and Clifford algebras, 

\[\mathcal{A}(\bb{R}^{2n\vert r})=\mrm{Weyl}(T^*\bb{R}^n)\otimes\mrm{Cliff}(\bb{R}^r,Q).\]

\subsubsection{Comparison to other Deformation Quantizations}

There have been many approaches to globalizing the canonical choice of local deformation quantization \cite{Fedosov,Deligne,Tomassini,FFSh,Willwacher,BK}.
In the super-case Bordemann \cite{Bordemann1,Bordemann2} constructed a deformation quantization for symplectic supermanifolds using Fedosov's approach. 
A similar result also appears in \cite{Engeli} using the methods of \cite{FFSh}. 
Our method of proof is similar, and inspired by, 
the methods used by Bezrukavnikov and Kaledin in the non-super case \cite{BK}. 
The formalism in \cite{BK} that we mimic here also works in the algebraic setting, 
and has even been extended to positive characteristic \cite{BK2}. 

Here, we globalize the local quantization and prove Theorem \ref{thm-1}, 
using techniques in formal geometry, first described by Gelfand-Kazhdan in \cite{GK},  
called \emph{Gelfand-Kazhdan descent}. 
This is a special case of Harish-Chandra descent. 
Roughly speaking, this is a fancy version of the Borel construction that takes into account the connection data on the formal coordinate bundle (\cite[\S 3.1]{BK}, \cite[\S 4.2]{Jordan}, or \S \ref{subsubsec-FormalCoord} in the super-symplectic case below). 
We develop Gelfand-Kazhdan descent for symplectic supermanifolds in \S\ref{subsec-Descent}. 
Gelfand-Kazhdan descent is also used in a more modern computation of 
the Witten genus coming from the factorization algebra of chiral differential operators \cite{GGW}. 
One benefit of using Gelfand-Kazhdan descent here is to make connections to Feynmann diagram computations in the BV formalism (as in \cite{GGW}) more accessible, see \S\ref{subsec-physics}. 

\begin{rmk}\label{rmk-BKours}
Essentially, we construct deformation quantizations locally on a formal disk and use gluing data $\msf{Glue}$ to descend to a deformation quantization on the whole manifold. 
In \cite{BK}, they classify the set $Q(M,\omegans)$ 
of deformation quantizations of a (non-super) symplectic $2n$-manifold $(M,\omegans)$ up to isomorphism using these techniques, \cite[Lem. 3.4]{BK}. 
This is done by describing the set of all possible gluing datum, 
which involves considering the unwieldy pro-group $\mrm{Aut}(\mrm{Weyl}_{2n})$ of automorphisms of the Weyl algebra, 
\[\msf{Glue}_{\mrm{Aut}(\mrm{Weyl}_{2n})}\xrta{\sim} Q(M,\omegans)\]
Here, we instead restrict the gluing datum to be linear. 
In the purely even case, this corresponds to requiring the data to come from the symplectic group $\mrm{Sp}(2n)$. 
We get a factoring of the equivalence from \cite{BK}, 
\[\begin{xymatrix}
{
\msf{Glue}_{\mrm{Aut}(\mrm{Weyl}_{2n})}\arw^{\simeq}[r] & Q(M,\omegans)\\
\msf{Glue}_{\mrm{Sp}(2n)}\arw[u]\arw[ur] & 
}
\end{xymatrix}\]

The space $\msf{Glue}_{\mrm{Sp}(2n)}$ is equivalent to the space of $\hbar$-formal exponentials (Remark \ref{rmk-BKours2}), 
which is contractible, Lemmas \ref{lem-FiberContractible} and \ref{lem-FiberContractible2}. 
In summary, the added rigidity produces a contractible space of gluing data, 
and hence an essentially unique deformation quantization. 
See also Remark \ref{rmk-BKours3}.
\end{rmk}

\subsection{Motivation and Broader Perspective}

Our present work is motivated by a larger program to relate genera to partition functions of field theories.
In \S\ref{subsec-invariants}, we give a zoomed-out look at how this paper relates to  manifold invariants of interest. 
The relationship between Fedosov quantizations and algebraic index theorems is discussed in \S\ref{subsec-Index}. 
Lastly, in \S\ref{subsec-physics}, 
we discuss the physical interpretation of this broader picture. 
Studying the questions raised here is ongoing joint work with Owen Gwilliam and Brian Williams.

\subsubsection{Motivation: Manifold Invariants}\label{subsec-invariants}

Let $\msf{sMfld}^\mrm{Sp}$ denote the category of symplectic supermanifolds, see Definition \ref{def-sympsuper} below. 
There is a category $\msf{sGK}^{=,\hbar}$ (Variation \ref{var-sgk=h}), 
fibered over $\msf{sMfld}^\mrm{Sp}$, 
of pairs $(\bb{M},\sigma)$ of a symplectic supermanifold and an $\hbar$-formal exponential; 
that is, the necessary input data for Theorem \ref{thm-1}. 
Roughly speaking, Theorem \ref{thm-1} provides a lift to the functor of $\bb{R}$-valued smooth functions, 

\[\begin{xymatrix}
{
\left(\msf{sGK}^{=,\hbar}\right)^\mrm{op}\arw[d]\ar@{-->}[rr]^-{\mathcal{A}_{(-)}} & & \msf{Mod}_{\Omega^\bullet}(\msf{sAlg}_{\bb{R}[[\hbar]]})\arw[d]^{\hbar=0}\\
\left(\msf{sMfld}^\mrm{Sp}\right)^\mrm{op}\arw[rr]_-{\mathcal{O}_{(-)}} & & \msf{Mod}_{\Omega^\bullet}(\msf{sAlg}_{\bb{R}})
}
\end{xymatrix}\]

As mentioned above, given a manifold $X$ and a quadratic vector bundle $E\rta X$ with compatible connection, 
one can produce an even symplectic structure on the supermanifold $(\pi^*E)[1]$, 
where $\pi\colon T^*X\rta X$ is the projection map. 
We obtain a functor 

\[L_X\colon\msf{VB}_{/X}^{\mrm{quad},\nabla}\rta\left(\msf{sMfld}^\mrm{Sp}\right)^\mrm{op},\]

where $\msf{VB}_{/X}^{\mrm{quad},\nabla}$ (Example \ref{ex-Rothstein}) is the category with

\begin{itemize}

\item objects: triples $(E,g,\nabla)$ of vector bundles over $X$, 
equipped with a quadratic form and a compatible connection, and  

\item morphisms: a morphism $(E,g,\nabla)\rta(E',g',\nabla')$ is a map of vector bundles $f\colon E\rta E'$ that is a fiberwise isomorphism, intertwines the quadratic forms, and so that $f^*\nabla'=\nabla$. 
\end{itemize} 

Just as the cotangent bundle of an ordinary manifold has a canonical quantization,  
we will construct a deformation quantization for symplectic supermanifolds coming from $\msf{VB}_{/X}^{\mrm{quad},\nabla}$, see Remark \ref{rmk-VBsGK} and Lemma \ref{lem-Connhbar}. 
This is done by constructing a lift of $L_X$ to $\msf{sGK}^{=,\hbar}$,

\[\begin{xymatrix}
{
& \left(\msf{sGK}^{=,\hbar}\right)^\mrm{op}\arw[d]\\
\msf{VB}_{/X}^{\mrm{quad},\nabla}\arw[r]_-{L_X}\ar@{-->}[ur]^{\tilde{L}_X} & \left(\msf{sMfld}^\mrm{Sp}\right)^\mrm{op}
}
\end{xymatrix}\]

Composing the lift $\tilde{L}_X$ with the deformation quantization functor $\mathcal{A}$ over $T^*X$, 
we obtain a functor (Remark \ref{rmk-LiftDefo})

\[\widetilde{A}_X\colon \msf{VB}_{/X}^{\mrm{quad},\nabla}\rta\msf{Mod}_{\Omega^\bullet_{T^*X}}(\msf{sAlg}_{\bb{R}[[\hbar]]}).\]

Further post-composing with the Hochschild cohomology functor 
$\mrm{HH}^\bullet_{\bb{R}[[\hbar]]}(-;-^\vee)$, 
we obtain a functor 

\[I_X\colon\msf{VB}_{/X}^{\mrm{quad},\nabla}\rta\msf{Mod}_{\Omega^\bullet_{T^*X}}(\msf{Ch}_{\bb{R}[[\hbar]]}).\]

\begin{quest}
What invariant of quadratic vector bundles on $X$ does $I_X$ produce?
\end{quest}

One well-studied invariant of quadratic vector bundles is the \emph{Witt group}, \cite{Witt}. 
It is natural to ask how $I_X$ and $\mrm{Witt}(X)$ are related. 
In particular, the Witt group is obtained by quotienting by the hyperbolic quadratic forms. 
Since Hochschild (co)homology is invariant under Morita equivalence, 
one might expect that $\widetilde{A}_X$ sends vector bundles with hyperbolic quadratic forms to Morita trivial algebras.

\begin{quest}
How does $I_X$ behave under stabilization by vector bundles with hyperbolic quadratic forms? 
In particular, does $\widetilde{A}_X$ send hyperbolic vector bundles to Morita trivial superalgebras?
\end{quest}  

\begin{ex}
The answer to this question is ``yes" when $X$ is a point. 
In this case, we are considering the functor from vector spaces equipped with a quadratic form to superalgebras. 
The functor $\widetilde{A}_\mrm{pt}$ sends a quadratic vector space $(V,Q)$ to the Clifford algebra $\mrm{Cliff}(V,Q)$. 
When $Q$ is hyperbolic, the Clifford algebra is equivalent to a matrix algebra via the spinor representation, 
and hence is Morita trivial. 
\end{ex}

The Witt group $\mrm{Witt}(X)$ is closely related to the (quadratic) L-groups, $\bb{L}(X)_\bullet$, \cite{RanickiAlg}. 
The L-groups are the natural home for the signature of $X$. 
As noted below (\S\ref{subsec-Index}), we expect a super-version of the algebraic index theorem (\cite{Engeli}) applied to certain oriented vector bundles over $X$ to recover the L-genus. 
There are also indications in the literature \cite{Berwick-Evans} that the 1d AKSZ theory relevant to 
$\widetilde{A}_X$ 
has partition function related to the L-genus (\S\ref{subsec-physics}). 
The invariant $I_X$ constructed here should therefore lead to interesting connections between super deformation quantization and the L-genus. 

\subsubsection{Motivation: Index Theory}\label{subsec-Index}
An essential invariant of a differential operator is its \emph{index}. 
One can ask how much the deformation $\mrm{Rees}(\mrm{Diff}_X)$ of $T^*X$ knows about the topology of $X$. 
Famously, Atiyah and Singer \cite{AtiyahSinger} 
proved that the (analytic) index of an elliptic differential operator $\mathcal{D}$ on $X$ is equivalent to its topological index.
Bressler, Nest, and Tsygan have proven an algebraic index theorem \cite{Nest-Tsygan, FFSh} using deformation theory. 
The algebraic index theorem, 
equips the Fedosov quantization of $(M,\omega,\nabla)$ (a symplectic manifold with symplectic connection $\nabla$) with an interesting trace map $\mrm{Tr}_M$, and then gives a description of the trace evaluated at 1 involving known topological invariants,

\[\mrm{Tr}_M(1)=\int_M\hat{A}(TM)\mrm{exp}(-\mrm{char}(\nabla)/\hbar).\footnote{There is a scalar term here, which depends on a normalization condition. Here $\mrm{char}(\nabla)$ is the \emph{characteristic class} of the deformation. See \cite[\S 4]{FFSh}.}\] 

In \cite{Engeli}, Engeli proves a generalization of the algebraic index theorem of Bressler-Nest-Tsygan \cite{Nest-Tsygan, FFSh} for certain symplectic supermanifolds of type $(2n\vert n,n)$.  
In Engeli's result \cite[Thm. 2.26]{Engeli}, 
one sees an invariant closely related to the multiplicative sequence for the L-genus  replacing the $\hat{A}$-genus in the non-super version. 
Our techniques of super-Gelfand-Kazhdan descent could be used to reproduce and generalize Engeli's super algebraic index result. 
See \S\ref{subsec-invariants} for more discussion along these lines. 

\subsubsection{Motivation: Quantum Field Theory}\label{subsec-physics}

The deformation quantization of $T^*\bb{R}^n$ is the Weyl algebra. 
In quantum mechanics, this is the algebra of observables of a free bosonic system.
The super-version, Theorem \ref{thm-1}, corresponds to adding fermions. 
The resulting Weyl-Clifford algebra is the algebra of local observables of suspersymmetric quantum mechanics. 

One can think of globalizing as going from the AKSZ theory for the formal super-disk to the theory for the symplectic supermanifold $\bb{M}$. 
On BV fields this is a process
\[\mrm{Maps}(S^1,\hat{\bb{D}}^{2n\vert r})\rightsquigarrow\mrm{Maps}(S^1,\bb{M}).\]

In \cite{GwilliamGrady}, Gwilliam and Grady construct 1d Chern-Simons theory in the BV formalism following Costello-Gwilliam \cite{CG1,CG2}. 
This 1d theory has quantum observables that agree with the Fedosov quantization of $T^*X$.  
We expect a super-analogue to \cite{GwilliamGrady} to show that 
the super-Fedosov quantization from Theorem \ref{thm-1} 
appears as the observables of supersymmetric quantum mechanics. 
Gelfand-Kazhdan descent for factorization algebras of observables has been developed in \cite{GGW}. 
Assuming one uses these descent techniques to describe supersymmetric quantum mechanics in the BV formalism, 
our proofs of Theorem \ref{thm-1} below 
make one well-positioned to compare the algebraic and physical constructions. 
Such a comparison for 1d Chern-Simons theory is made in \cite{GradyLiLi,newSiLi}. 

\subsection{Linear Overview}

We give a brief overview of the structure of this paper. 

In \S\ref{sec-Review}, we review the basics of symplectic supermanifold, 
including Rothstein's analogue \cite{Rothstein} of Batchelor's structure theorem for supermanifolds \cite{Batchelor}. 
We define super-Harish-Chandra pairs and construct the particular example of such that we will use for our descent in \S\ref{subsec-HCpair}. 
Our descent functor is defined in \S\ref{subsec-Descent}, 
where we also give a few first examples of how descent works. 
In \S\ref{subsec-Descent}, we also prove several monoidal properties of our super-Gelfand-Kazhdan descent functor. 
The formalism developed in \S\ref{sec-GKDescent} is used in \S\ref{sec-DefQuant} and \S\ref{subsec-New} to show that Gelfand-Kazhdan descent takes deformation quantizations to deformation quantizations. 
In \S\ref{sec-SuperFedosov}, we prove Theorem \ref{thm-1}, giving a deformation quantization of a symplectic supermanifold. 
We then describe the deformation quantization in terms of Weyl and Clifford algebras in \S\ref{subsec-WeylClifford}. 

\subsection{Conventions}
We set the following conventions for the paper. 

\noindent\textit{Algebra Conventions}.
\begin{itemize}
\item Let $\bb{k}$ be either $\bb{R}$ or $\bb{C}$
\item $K$ will be a Lie supergroup
\item $\mfrk{g}$ is a Lie superalgebra 
\end{itemize}

\noindent\textit{Manifold Conventions}.

All manifolds are real (i.e., not complex), smooth and without boundary. 
A \emph{manifold with boundary} is a manifold with, possibly empty, boundary. 
We use the phrase ``ordinary manifold" to distinguish from a supermanifold.
\begin{itemize}
\item $X$ will denote an ordinary manifold
\item $\bb{X}$ will denote a supermanifold
\item $(M,\omegans)$ will denote a ordinary symplectic manifold
\item $(\bb{M},\omega)$ will denote a symplectic supermanifold
\item $\mathcal{O}_Y$ denotes smooth $\bb{k}$-valued functions on $Y$
\item $\widehat{\bb{D}}^{2n\vert r}$ is the formal super-disk of dimension $2n\vert r$ whose ring of functions is 
\[\mathcal{O}_{\widehat{\bb{D}}^{2n\vert r}}=\bb{k}[[p_1,\dots,p_n,q_1,\dots,q_n,\theta_1,\dots,\theta_r]].\]
\item Given a $\bb{k}$-vector space $V$, the trivial vector bundle on $\bb{X}$ with fiber $V$ is denoted $\underline{V}_\bb{X}$.
\end{itemize}

\noindent Further conventions are explained later, see Convention \ref{conv-HC}.

\subsection{Acknowledgements}
We thank Owen Gwilliam and Brian Williams for introducing us to the field, 
and for numerous helpful discussions, 
both in mathematical content and inspiration. 
Additional thanks are due to Dan Berwick-Evans for useful conversations, and to our advisor Michael Hopkins for help along the way. 
Thanks to Bertram Arnold for pointing out a mistake in applying Lemma 4.10 from a previous version of this paper. 

The author was supported by NSF Grant No. 1122374 while completing this work.

\section{Review of Symplectic Supermanifolds}\label{sec-Review}
We review the basics of symplectic supermanifolds that we will use below. 
For more comprehensive discussions of supermanifolds, see \cite{Batchelor, Leites, Rogers}.

\begin{defn}
A \emph{supermanifold} is a $\bb{Z}/2$-graded ringed space $\bb{X}$ 
whose underlying space $X$ is an $n$-manifold and such that the \emph{ring of smooth functions} $\mathcal{O}_\bb{X}$ is locally isomorphic to 
\[\mathcal{O}_{\bb{R}^n}\otimes\Lambda[\theta_1,\dots, \theta_r]\]
for some $r$. 
A \emph{morphism of supermanifolds} $\bb{X}\rta\bb{Y}$ is a graded map 
$\cal{O}_\bb{Y}\rta\cal{O}_\bb{X}$ living over a smooth map $X\rta Y$.
\end{defn}

We will let $\msf{sMfld}_{n\vert r}$ denote the category of supermanifolds 
with $n$ even dimensions and $r$ odd dimensions.

\begin{ex}[Ordinary manifolds as supermanifolds]
Let $X$ be an ordinary (i.e., not super) $n$-manifold. 
We can regard $X$ as a supermanifold with $0$ odd directions.
\end{ex}

\begin{ex}
We let $\bb{R}^{n\vert r}$ denote the supermanifold with underlying manifold $\bb{R}^n$
and functions
\[\cal{O}_{\bb{R}^{n\vert r}}=\bb{R}[x_1,\dots,x_n]\otimes\Lambda[\theta_1,\dots,\theta_r].\]
\end{ex}

\begin{ex}[Batchelor's theorem]
Let $X$ be an ordinary $n$-manifold and $E\rta X$ a rank $k$ vector bundle on $X$.
One can form a supermanifold $E[1]$ with underlying ordinary manifold $X$ and functions
\[\cal{O}_{E[1]}=\Gamma(X,\Lambda^\bullet E^\vee).\]
For the tangent bundle $TX\rta X$, we use the notation $T[1]X$, 
which has $\cal{O}_{T[1]X}=\Omega^\#_X$,
the underlying $\bb{Z}/2$-graded vector space of the de Rham complex. 
The notation $\Pi E$ is sometimes used for $E[1]$. 
By Batchelor's theorem \cite[\S 3]{Batchelor}, 
every supermanifold is noncanonically isomorphic to one of the form $E[1]$.
\end{ex}

\begin{defn}\label{def-VFsympsuper}
Given a supermanifold $\bb{X}$, \emph{vector fields} on $\bb{X}$ is the Lie superalgebra
\[\mrm{Vect}(\bb{X})\colon=\mrm{Der}(\cal{O}_\bb{X})\]
of graded derivations.

\emph{Forms of degree $k$} on $\bb{X}$ is the space
\[\Omega^k_\bb{X}\colon=\Lambda^k(\mrm{Vect}(\bb{X})^\vee).\]

The \emph{de Rham complex} of $\bb{X}$ is $\Omega^\bullet_\bb{X}$ with differential defined to be the derivation of bidegree $(1,\text{even})$ which locally on generators $x_i,\theta_j,dx_i,d\theta_j$ is given by $d(x_i)=dx_i$, $d(\theta_i)=d\theta_i$, $d(dx_i)=0$, and $d(d\theta_i)=0$.

\end{defn}

\noindent Note that $\Omega^\bullet_\bb{X}$ inherits a $\bb{Z}/2$-grading, 
so we can speak of even and odd forms on $\bb{X}$.

\begin{defn}\label{def-sympsuper}
A \emph{symplectic supermanifold} is a pair $(\bb{M},\omega)$ where $\bb{M}$ is a supermanifold and 
$\omega$ is an even, closed, non-degenerate 2-form on $\bb{M}$.

A \emph{symplectomorphism} $\psi\colon(\bb{M},\omega)\rta(\bb{M}',\omega')$ is
a morphism of supermanifolds $\bb{M}\rta\bb{M}'$ that is a diffeomorphism on underlying manifolds
and so that $\psi^*(\omega')=\omega$.
\end{defn}

We let $\msf{sMfld}_{n\vert r}^\mrm{Sp}$ denote the category of symplectic supermanifolds
with $n$ even dimensions and $r$ odd dimensions.

\begin{ex}[Ordinary symplectic manifolds as symplectic supermanifolds]
Let $(M,\omegans)$ be an ordinary symplectic manifold. 
Viewing $M$ as a supermanifold with 0 odd directions,
$\omegans$ becomes an even 2-form.
Thus $(M,\omegans)$ can be seen as a symplectic supermanifold.
\end{ex}

\begin{ex}\label{ex-LocalSymplecticStructure}
By Darboux's theorem, \cite[Thm. 8.1]{CannasdaSilva}, 
every (ordinary) symplectic manifold $(M,\omegans)$ is locally isomorphic to $(\bb{R}^{2n},\omega_0)$
where, in coordinates $p_1,\dots,p_n,q_1,\dots,q_n$, 
the form $\omega_0$ is
\[\omega_0=\sum_{i=1}^n p_i\wedge q_i.\]
We can similarly give $\bb{R}^{2n\vert r}$ a symplectic structure
but we need to make a choice of $Q=(\ep_1,\dots,\ep_r)$ where $\ep_i\in\{1,-1\}$. 
Given such a $Q$, we can define a symplectic form on $\bb{R}^{2n\vert r}$ by
\[\omega_Q=\sum_{i=1}^n dp_i\wedge dq_i+\sum_{i=1}^r\frac{\ep_i}{2}d\theta_i^2\]
where $\theta_1,\dots,\theta_r$ are the odd coordinates. 
Note that $\omega_Q$ is equivalent to the date of its \emph{signature}, 
that is the number $a$ of positive $\ep_i$ and the number $b$ of negative $\ep_i$. 
We have $a+b=r$. 
\end{ex}

\begin{notation}
We denote the symplectic supermanifold described in Example \ref{ex-LocalSymplecticStructure} by $(\bb{R}^{2n\vert r},\omega_Q)$. 
\end{notation}

\begin{defn}\label{def-symportho}
A \emph{symplectic super vector space} is a super vector space $V$ 
together with a nondegenerate bilinear form $b\colon V\times V\rta\bb{k}$ that is skew-symmetric in the even directions and symmetric in the odd directions, 
\[b(x,y)=(-1)^{|x||y|}b(y,x).\]
\end{defn}

Let $(V,b)$ be a symplectic super vector space of dimension $2n\vert r$. 
Let $Q$ be the quadratic form associated to the nondegenerate bilinear form $b$. 
Analogously to the purely even case \cite[Thm. 1.1]{CannasdaSilva}, there is an isomorphism between $(V,b)$ and $(\bb{R}^{2n\vert r},\omega_Q)$ from Example \ref{ex-LocalSymplecticStructure}. 

\begin{defn}
The {$(2n\vert a,b)$-symplectic group}, denoted $\mrm{Sp}(2n\vert a,b)$, 
is the group of linear symplectomorphisms of $(\bb{R}^{2n\vert r},\omega_Q)$ 
where $Q$ has signature $(a,b)$. 
\end{defn}

When $(a,b)=(r,0)$, this is sometimes called the symplectic-orthogonal group. 

\begin{rmk}\label{rmk-sp2nab}
Define the super-transpose of a block matrix 
\[M=
\begin{bmatrix}
A & B\\
C & D
\end{bmatrix}\]
to be
\[M^{sT}=
\begin{bmatrix}
A^T & -C^T\\
B^T & D^T
\end{bmatrix}.\]
Let $G$ denote the diagonal matrix $G=\mrm{diag}(\ep_1,\dots,\ep_r)$ where $Q=(\ep_1,\dots,\ep_r)$. 
If $\mrm{GL}(2n\vert r)$ denotes the general linear supergroup, 
then $\mrm{Sp}(2n\vert a,b)\subset\mrm{GL}(2n\vert r)$ consists of those matrices $M$ so that
\[M^{sT}HM=H\]
where 
\[H=
\begin{bmatrix}
\Omega & 0\\
0 & G
\end{bmatrix}\]
and 
\[\Omega=
\begin{bmatrix}
0 & \mrm{Id}_n\\
-\mrm{Id}_n & 0
\end{bmatrix}.\]

\end{rmk}

Note that when there are no odd dimensions we have $\mrm{Sp}(2n\vert 0,0)\cong\mrm{Sp}(2n)$.

\begin{rmk}
If we replace $G$ with a conjugate matrix, 
we obtain an isomorphic Lie group. 
\end{rmk}

\begin{ex}\label{ex-Rothstein}
Let $(M,\omegans)$ be an ordinary symplectic manifold, $E\rta M$ a vector bundle,
and $(g,\nabla)$ a metric and compatible connection on $E$. 
The data $(\omegans,g,\nabla)$ defines a super-symplectic form on the supermanifold $E[1]$, see \cite[Def. 1]{Rothstein}. 
Let $\msf{VB}^{\mrm{quad},\nabla}_{/\msf{Mfld}^\mrm{Sp}}$ be the category of quadruples $(M,E,g,\nabla)$ and morphisms respecting this data. 
Explicitly, a morphism $(M,E,g,\nabla)\rta(M',E',g',\nabla')$ is a map of vector bundles $f\colon E\rta E'$ that is a fiberwise isomorphism, lives over a local symplectomorphism $M\rta M'$, intertwines the quadratic forms, and so that $f^*\nabla'=\nabla$. 
\end{ex}

There is a symplectic analogue of Batchelor's theorem \cite{Batchelor}, due to Rothstein, \cite{Rothstein}.

\begin{thm}[Rothstein]\label{thm-Rothstein}
Every symplectic supermanifold is non-canonically isomorphic to one of the form in Example \ref{ex-Rothstein}.
\end{thm}

\begin{cor}
Let $(\bb{M},\omega)$ be a symplectic supermanifold. 
Then the underlying manifold $M$ inherits the structure of an ordinary symplectic manifold. 
In particular, $\bb{M}$ must have an even number of even directions.
\end{cor}

We also have a super-analogue of the Darboux theorem \cite[Thm. 5.3]{Kostant}

\begin{thm}[Kostant]\label{thm-SuperDarboux}
Let $(\bb{M},\omega)\in\msf{sMfld}_{2n\vert r}^\mrm{Sp}$. 
Then there exists $Q=(\ep_1,\dots,\ep_r)$ so that for every $x\in\bb{M}$,
there exists an open neighborhood $U$ of $x$ that is symplectomorphic to $(\bb{R}^{2n\vert r},\omega_Q)$.
\end{thm}

Here $\omega_Q$ is as in Example \ref{ex-LocalSymplecticStructure}. 
See also \cite[\S 3]{Tuynman} and the references therein. 

\begin{notation}
Let $(\bb{M},\omega)$ be a symplectic supermanifold so that locally $\omega$ is of the form $\omega_Q$ with signature $(a,b)$. 
We refer to $(\bb{M},\omega)$ as having \emph{type} $(2n\vert a,b)$.
\end{notation}

\begin{defn}
Let $\msf{sMfld}_{2n\vert a,b}$ denote the category of symplectic supermanifolds of type $(2n\vert a,b)$ and local symplectomorphisms. 
\end{defn}

Next, we would like to discuss symplectic vector fields on a symplectic supermanifold. 
For motivation and to review, we first recall the notions on ordinary symplectic manifolds. 

\subsubsection{Symplectic Vector Fields: Ordinary Manifolds}

Let $(M,\omegans)$ be a symplectic manifold. 
The nondegenerate 2-form $\omegans$ determines an isomorphism $TM\cong T^*M$, 
and thus an equivalence 
\[\phi_{\omegans}\colon \mrm{Vect}(M)\cong \Omega^1(M).\] 

\begin{defn}\label{def-SympVF}
Let $(M,\omegans)$ be a symplectic manifold. 
The Lie algebra of \emph{symplectic vector fields} is the sub-Lie algebra of 
$\mrm{Vect}(M)$ consisting of those vector fields $v$ such that 
$\phi_{\omegans}(v)$ is closed. 
Denote by $\mrm{Vect}^\mrm{symp}(M)$ the Lie algebra of symplectic vector fields.

Say $v$ is a \emph{Hamiltonian vector field} if $\phi_{\omegans}(v)$ is exact. 
In this case, we refer to a function $h$ such that $dh=\phi_{\omegans}(v)$ as a \emph{Hamiltonian} of $v$. 
\end{defn}

We will describe a characterization of symplectic vector fields in terms of the ring of functions $\mathcal{O}_M$. 
To do this, we need to understand the structure the symplectic form $\omegans$ induces on $\mathcal{O}_M$. 

\begin{defn}
A \emph{Poisson algebra} is an commutative algebra $P$ equipped with a Lie bracket $\{-,-\}$ satisfying the Leibnitz rule 
\[\{f,gh\}=\{f,g\}h+g\{f,h\}\]
for any $f,g,h\in P$. 

A \emph{Poisson derivation} of $P$ is a linear map $d\colon P\rta P$ so that for all $x,y\in P$ we have
\begin{itemize}
\item $d(xy)=d(x)y+xd(y)$, and 
\item $d(\{x,y\})=\{d(x),y\}+\{x,d(y)\}$.
\end{itemize}
\end{defn}

The following is \cite[Lem. 1.1.18]{McDuffSalamon}.
\begin{lem}
Let $(M,\omegans)$ be a symplectic manifold. 
Then $\mathcal{O}_M$ is a Poisson algebra with bracket 
\[\{f,g\}_{\omegans}=\phi_{\omegans}^{-1}(df)(g).\]
Here, $\phi_{\omegans}^{-1}(df)$ is the Hamiltonian vector field with Hamiltonian $f$.
\end{lem}

\begin{notation}
Let $(M,\omegans)$ be a symplectic manifold. 
We let $\mrm{Der}_{\omegans}(\mathcal{O}_M)$ denote the Lie algebra of Poisson derivations of the Poisson algebra $(\mathcal{O}_M,\{-,-\}_{\omegans})$. 
\end{notation}

The following is \cite[Def. 18.2]{CannasdaSilva}. 

\begin{lem}\label{lem-SympVFasDer}
Let $(M,\omegans)$ be a symplectic manifold. 
There is an equivalence of Lie algebras
\[\mrm{Vect}^\mrm{symp}(M,\omegans)\simeq\mrm{Der}_{\omegans}(\mathcal{O}_M).\]
\end{lem}

\subsubsection{Symplectic Vector Fields: Supermanifolds}

For the super case, we mimic the description of symplectic vector fields as derivations of a Poisson algebra.

\begin{defn}
A \emph{Poisson superalgebra} is a supercommutative superalgebra $R$ equipped with a Lie superbracket $\{-,-\}$ such that 
\[\{f,gh\}=\{f,g\}h+(-1)^{|f||g|}g\{f,h\}\]
for all $f,g,h\in R$.
\end{defn}

The following is in \cite[Pg. 244]{Tuynman}.

\begin{lem}\label{lem-OisPoissonSuperalg}\label{lem-superPoisson}
Let $(\bb{M},\omega)$ be a symplectic supermanifold. 
Then $\omega$ induces an equivalence 
\[\phi_\omega\colon\mrm{Vect}(\bb{M})\cong \Omega^1(\bb{M}),\]
and $\mathcal{O}_\bb{M}$ is a Poisson superalgebra under the superbracket
\[\{f,g\}_\omega=\phi_\omega^{-1}(df)(g).\]
\end{lem}

\begin{ex}\label{ex-FormalPoisson}
Since we will be using formal geometry, 
we will often be interested in the formal super-disk $\widehat{\bb{D}}^{2n\vert r}$. 
If we give $\bb{R}^{2n\vert r}$ a symplectic form of type $(2n\vert a,b)$, 
then functions on the formal disk inherits a Poisson algebra structure from the completion of functions on $\bb{R}^{2n\vert r}$ at the point 0. 
In coordinates, the Poisson bracket on 

\[\mathcal{O}_{\widehat{\bb{D}}^{2n\vert r}}=\bb{k}[[p_1,\dots,p_n,q_1,\dots,q_n,\theta_1,\dots,\theta_a,\theta'_1,\dots,\theta'_b]]\]

is given by 
\[\{p_i,q_i\}=1\]
\[\{\theta_i,\theta_j\}=1\]
\[\{\theta_i',\theta_j'\}=-1\]
and the rest zero. 
We denote this Poisson algebra by $\widehat{\mathcal{O}}_{2n\vert a,b}$. 
\end{ex}

\begin{notation}
Let $\mrm{Der}_\omega(\mathcal{O}_\bb{M})$ denote the Lie superalgebra of Poisson derivations of the Poisson superalgebra $(\mathcal{O}_\bb{M},\{-,-\}_\omega)$. 
\end{notation}

\begin{defn}
Let $(\bb{M},\omega)$ be a symplectic supermanifold. 
The Lie superalgebra of \emph{symplectic vector fields} on $\bb{M}$ is the Lie superalgebra of derivations  
\[\mrm{Vect}^\mrm{symp}(\bb{M},\omega)=\mrm{Der}_\omega(\mathcal{O}_\bb{M}).\]
\end{defn}

\subsubsection{Bundles}\label{sec-Bundles}

For an overview of vector bundles and principal bundles on supermanifolds, 
see \cite{Supergeometry} or \cite{Superbundles}. 
For a construction of the frame bundle of a supermanifold, see \cite[\S 2]{Tuynman}.
Just as the structure group of the frame bundle of an ordinary symplectic manifold can be reduced to the symplectic group, 
we have the following, 

\begin{lem}\label{lem-SymplecticFrame}
Let $(\bb{M},\omega)$ be a symplectic supermanifold 
of type $2n\vert a,b$. 
Then the structure group of the frame bundle $\mrm{Fr}_\bb{M}\rta\bb{M}$ 
can be reduced to $\mrm{Sp}(2n\vert a,b)$. 
\end{lem}
The fiber over $x\in\bb{M}$ will be
\[\mrm{Fr}^{\mrm{Sp}(2n\vert a,b)}_\bb{M}\vert_x=\mrm{Symp}\left((T_x\bb{M},\omega\vert_x),(\bb{R}^{2n\vert r},\omega_Q)\right),\]
the group of linear symplectomorphisms. 

\section{Gelfand-Kazhdan Descent for Symplectic Supermanifolds}\label{sec-GKDescent}

We would like to construct a descent functor that allows us to study symplectic supermanifolds locally. 
The notion of descent we will consider is a variant of the Borel construction. 
For $K$ a Lie group, 
the Borel construction takes a principal $K$-bundle $P\rta X$ and a $K$-module $V$ to the vector bundle $P\times_KV\rta X$. 
Harish-Chandra descent is a generalization of this construction from $K$ to a Harish-Chandra pair $(\mfrk{g},K)$. 
We will need a slightly more complicated version of Harish-Chandra descent, 
known as Gelfand-Kazhdan descent, \cite[Def. 2.17]{GGW}. 

\begin{rmk}
The Gelfand-Kazhdan descent considered here generalizes that in \cite[\S 2.4]{GGW} in two ways. 
First, we work with supermanifolds. 
Second, our descent allows for general symplectic manifolds rather than just cotangent bundles. 
The symplectic case is also studied in \cite{BK}. 
\end{rmk}

\subsection{Harish-Chandra Pair}\label{subsec-HCpair}

We define the Harish-Chandra pair we will use for our super-Gelfand-Kazhdan descent. 

\begin{defn}
A \emph{super-Harish-Chandra} pair (sHC pair) is a pair $(\mfrk{g}, K)$ 
where $\mfrk{g}$ is a Lie superalgebra and $K$ is a Lie supergroup together with
\begin{itemize}
\item an injective Lie superalgebra map $i\colon{\mrm{Lie}(K)}\rta\mfrk{g}$
\item an action of $K$ on $\mfrk{g}$, $\rho_K\colon K\rta\mrm{Aut}(\mfrk{g})$
\end{itemize}
such that the action of ${\mrm{Lie}(K)}$ on $\mfrk{g}$ induced by $\rho_K$,
\[\mrm{Lie}(\rho_K)\colon{\mrm{Lie}(K)}\rta\mrm{Der}(\mfrk{g}),\]
is the adjoint action induced from the embedding $i$.
\end{defn}

In the purely even case, 
when $\mfrk{g}$ is an ordinary Lie algebra and $K$ is an ordinary Lie group, 
this recovers the usual (non-super) definition of an HC pair.

\begin{defn}
A \emph{morphism of super-Harish-Chandra pairs} $(\mfrk{f},f)\colon(\mfrk{g},K)\rta (\mfrk{g}',K')$ is

\begin{itemize}
\item a map of Lie superalgebras $\mfrk{f}\colon\mfrk{g}\rta\mfrk{g}'$ and
\item a map of Lie supergroups $f\colon K\rta K'$
\end{itemize}
such that the diagram of Lie superalgebras

\[\begin{xymatrix}
{
{\mrm{Lie}(K)}\arw[r]^{\mrm{Lie}(f)}\arw[d]_i & {\mrm{Lie}(K)}'\arw[d]^{i'}\\
\mfrk{g}\arw[r]_{\mfrk{f}} & \mfrk{g}'
}
\end{xymatrix}\]

commutes. 
\end{defn}

\begin{ex}
The category of sHC pairs has an initial object $(0,e)$, where $0$ is the 0-dimensional Lie superalgebra and $e$ is the $0$-dimensional Lie supergroup consisting of the identity point.
\end{ex}

\begin{ex}\label{ex-BasicHC}
Let $K$ be a Lie supergroup. 
Then $({\mrm{Lie}(K)}, K)$ is sHC pair, see \cite[Thm. 3.5]{Kostant}. 
\end{ex}

\begin{ex}\label{ex-HCpairSubgroup}
Let $G$ and $K$ be Lie supergroups. 
Let $\iota \colon K'\subset K$ be a closed sub-supergroup. 
There is a unique sHC pair structure on $(\mrm{Lie}(G),K')$ so that 
\[(\mrm{Id}_{\mrm{Lie}(G)}, \iota)\colon (\mrm{Lie}(G),K')\rta(\mrm{Lie}(G), K)\]
is a morphism of super-Harish-Chandra pairs. 
This is \cite[Ex. 1.2]{GGW}. 

More generally, if $(\mfrk{g},K)$ is an sHC pair and $K'\subset K$ is a closed sub-supergroup, 
then there is a unique sHC pairs structure on $(\mfrk{g}, K')$ so that $(\mrm{Id}_\mfrk{g},\iota)$ is a morphism of sHC pairs.
\end{ex}

\begin{ex}
Let $(\mfrk{g},K)$ be an HC pair. 
Given a central extension $\hat{\mfrk{g}}$ of $\mfrk{g}$ that is split over $\mrm{Lie}(K)$, 
the pair $(\hat{\mfrk{g}},K)$ is an HC pair. 
This is in \cite[\S 2.1.1]{BenZviFrenkel}. 
\end{ex}

The following lemmas allows us to produce more examples of sHC pairs. 

\begin{lem}\label{lem-HCsubLie}
Let $(\mfrk{g}, K)$ be an sHC pair. 
Let $j\colon \mfrk{g}'\subset \mfrk{g}$ be a sub-Lie superalgebra. 
If the injective map ${\mrm{Lie}(K)}\rta\mfrk{g}$ factors through a map $j'\colon{\mrm{Lie}(K)}\rta \mfrk{g}'$, 
then there is a unique sHC pair structure on $(\mfrk{g}',K)$ so that $(j', \mrm{Id}_K)$ is a morphism of sHC pairs.
\end{lem}

\begin{proof}
Since ${\mrm{Lie}(K)}\rta\mfrk{g}$ is injective, 
so is the factored map $j'\colon{\mrm{Lie}(K)}\rta\mfrk{g}'$. 
To produce an action $\rho_K'\colon K\rta\mrm{Aut}(\mfrk{g}')$ of $K$ on $\mfrk{g}$', 
note that the adjoint action of ${\mrm{Lie}(K)}$ on $\mfrk{g}$ 
(via the embedding $i\colon{\mrm{Lie}(K)}\rta\mfrk{g}$)
may be restricted to the adjoint action of ${\mrm{Lie}(K)}$ on $\mfrk{g}'$ 
(via the embedding $j$). 
Since $(\mfrk{g},K)$ is an sHC pair, 
the adjoint action of ${\mrm{Lie}(K)}$ on $\mfrk{g}$ is given by $\mrm{Lie}(\rho_K)$. 
Thus, for $x\in{\mrm{Lie}(K)}$ and $g\in\mfrk{g}'$, 
the adjoint action of $x$ on $g$ is given by the formula 

\[x(g)=\frac{d}{dt}\vert_{t=0}\exp(tx)\cdot g.\]

This is exactly the formula for the derivative of an action 
$\rho_K'\colon K\rta\mrm{Aut}(\mfrk{g}')$. 
Thus $(\mfrk{g}',K)$ is an sHC pair. 
The pair $(j',\mrm{Id}_K)$ is a morphism of sHC pairs by construction. 
\end{proof}

The following geometric incarnation of Example \ref{ex-BasicHC} 
will be the motivation from which we will construct our sHC pair of interest. 

\begin{ex}
Let $\bb{X}$ be a supermanifold. 
Then $(\mrm{Vect}(\bb{X}),\mrm{Diff}(\bb{X}))$ is almost an sHC pair. 
As $\mrm{Diff}(\bb{X})$ is infinite-dimensional, this is not technically an example of an sHC pair. 
However, there is an injective map 
\[\mrm{Diff}(\bb{X})\rta\mrm{Aut}(\mathcal{O}_\bb{X})\]
and one can think of the Lie algebra of $\mrm{Aut}(\mathcal{O}_{\bb{X}})$ as being vector fields on $\bb{X}$, 
\[\mrm{Vect}(\bb{X})= \mrm{Der}(\mathcal{O}_\bb{X})``\simeq"\mrm{Lie}(\mrm{Aut}(\mathcal{O}_\bb{X})).\]

If $\bb{X}$ is an affine space, we can make a related precise statement. 
The linear diffeomorphisms $\mrm{GL}(\bb{X})$ of $\bb{X}$ form a sub-supergroup of $\mrm{Aut}(\mathcal{O}_\bb{X})$. 
Now, $\mrm{GL}(\bb{X})$ is a Lie supergroup, and 
$(\mrm{Vect}(\bb{X}),\mrm{GL}(\bb{X}))$ 
is an sHC pair.
\end{ex}

\begin{ex}
Let $\bb{M}=\bb{R}^{2n\vert r}$ be the symplectic supermanifold with symplectic form $\omega_Q$, as in Example \ref{ex-LocalSymplecticStructure}. 
Then $\mrm{Sp}(2n\vert a,b)$ is a closed sub-supergroup of $\mrm{GL}(\bb{M})$, 
where $(a,b)$ is the signature of $Q$, see Remark \ref{rmk-sp2nab}.
By Example \ref{ex-HCpairSubgroup}, we get an sHC pair 
\[(\mrm{Vect}(\bb{R}^{2n\vert r}) ,\mrm{Sp}(2n\vert a,b)).\] 
\end{ex}

Recall the Lie sub-superalgebra of symplectic vector fields 
$j\colon \mrm{Vect}^\mrm{symp}(\bb{M},\omega )\subset\mrm{Vect}(\bb{M})$ 
from Definition \ref{def-SympVF}.

\begin{cor}\label{cor-HCpairSymp}
There is a unique sHC pair structure on $(\mrm{Vect}^\mrm{symp}(\bb{M},\omega),\mrm{Sp}(2n\vert a,b))$ so that $(j,\mrm{Id})$ is a morphism of sHC pairs.
\end{cor}

\begin{proof}
This follows from Lemma \ref{lem-HCsubLie}.
\end{proof}

We would like to mimic the above story for $\bb{R}^{2n\vert r}$ replaced with the formal super-disk. 

\begin{defn}\label{def-formalnbhd}
Let $\bb{X}$ be a supermanifold and $x\in \bb{X}$. 
Let $\mathcal{O}_{\bb{X},x}$ denote the superalgebra of germs of functions at $x$. 
Let $\mfrk{m}_x$ be the ideal of functions vanishing at $x$. 
The ring of functions of the \emph{formal neighborhood} of $\bb{X}$ at $x$ is the limit 

\[\widehat{\mathcal{O}}_{\bb{X},x}=\lim_i \mathcal{O}_{\bb{X},x}/\mfrk{m}_x^i.\] 

The Lie superalgebra of \emph{formal vector fields} of $\bb{X}$ at $x$ is the Lie superalgebra of derivations

\[\widehat{\mrm{Vect}}_x(\bb{X})=\mrm{Der}(\widehat{\mathcal{O}}_{\bb{X},x}).\]
\end{defn}

\begin{rmk}
One can consider a subgroup $\mrm{Aut}^\mrm{filt}(\widehat{\mathcal{O}}_{\bb{X},x})$ 
of $\mrm{Aut}(\widehat{\mathcal{O}}_{\bb{X},x})$ of filtration preserving automorphisms. 
This subgroup can be considered as a pro-Lie supergroup. 
Its Lie superalgebra consists of those vector fields that vanish at $x$. 
We have a non-canonical equivalence

\begin{align}\label{eq-VanishAtZero}
\mrm{Lie}(\mrm{Aut}^\mrm{filt}(\widehat{\mathcal{O}}_{\bb{X},x}))\oplus \bb{R}^{2n\vert r}\simeq \widehat{\mrm{Vect}}_x(\bb{X}),
\end{align}

where $\bb{X}$ has dimension $2n\vert r$. 
This equivalence is given informally by taking a pair $(v_0,y)$ 
of a vector field that vanishes at $x$ and a vector $y\in\bb{R}^{2n\vert r}$ 
to the vector field that looks like $v_0$ translated by $y$. 
See \cite[\S 2.1.1]{GGW} for the non-super analogue.

\end{rmk}

\begin{lem}
The pair 
$(\widehat{\mrm{Vect}}_x(\bb{X}),\mrm{Aut}^\mrm{filt}(\widehat{\mathcal{O}}_{\bb{X},x}))$ 
has the structure of a pro-sHC pair.
\end{lem}

\begin{proof}
This follows from the pro-version of Example \ref{ex-HCpairSubgroup}. 
Indeed, 
$\widehat{\mrm{Vect}}_x(\bb{X})=\mrm{Der}(\widehat{\mathcal{O}}_{\bb{X},x})$ 
is the Lie algebra of $\mrm{Aut}(\widehat{\mathcal{O}}_{\bb{X},x})$, 
and filtration preserving automorphisms are a sub-pro-Lie group. 
\end{proof}

As in the non-formal case, we would like to restrict to symplectic vector fields and linear symplectomorphisms. 

\begin{lem}\label{lem-PoissonSuperalgebra}
Let $(\bb{M},\omega)$ be a symplectic manifold and $x\in \bb{M}$ a point. 
The Poisson superalgebra structure on $\mathcal{O}_\bb{M}$ induces a Poisson superalgebra structure on $\widehat{\mathcal{O}}_{\bb{M},x}$.
\end{lem}

\begin{proof}
This follows from \cite[\S 1.2]{EtingofKazhdan}. 
\end{proof}

In analogy with Lemma \ref{lem-SympVFasDer}, we make the following definition.

\begin{defn}
Let $(\bb{M},\omega)$ be a symplectic supermanifold and $x\in \bb{M}$ a point. 
The Lie superalgebra of \emph{symplectic formal vector fields} on $\bb{M}$ at $x$ is 
the Lie superalgebra of Poisson derivations of the Poisson superalgebra 
$(\widehat{\mathcal{O}}_{\bb{M},x},\{-,-\}_\omega)$, 
\[\widehat{\mrm{Vect}}^\mrm{symp}_x(\bb{M},\omega)\colon =\mrm{Der}_{\omega}(\widehat{\mathcal{O}}_{\bb{M},x}).\]
\end{defn}

Consider the case when $\bb{M}$ is affine. 
That is, take a symplectic supermanifold of the form $(\bb{R}^{2n\vert r},\omega_Q)$ from Example \ref{ex-LocalSymplecticStructure}. 
We have a formal version of Corollary \ref{cor-HCpairSymp}. 

\begin{lem}
There is a unique sHC pair structure on 
\[(\widehat{\mrm{Vect}}_0^\mrm{symp}(\bb{R}^{2n\vert r},\omega_Q),\mrm{Sp}(2n\vert a,b))\]
so that the inclusion of symplectic vector fields and the inclusion of linear symplectic automorphisms induce morphisms of sHC pairs

\begin{align*}
\left(\widehat{\mrm{Vect}}_0^\mrm{symp}(\bb{R}^{2n\vert r},\omega_Q),\mrm{Sp}(2n\vert a,b)\right)&\rta  \left(\widehat{\mrm{Vect}}_0(\bb{R}^{2n\vert r},\omega_Q),\mrm{Sp}(2n\vert a,b)\right)\\
&\rta \left(\widehat{\mrm{Vect}}_0(\bb{R}^{2n\vert r},\omega_Q), \mrm{Aut}^\mrm{filt}(\widehat{\mathcal{O}}_{\bb{R}^{2n\vert r},0})\right).
\end{align*}

\end{lem}

\begin{proof}
This follows from the pro-version of Lemma \ref{lem-HCsubLie} and Example \ref{ex-HCpairSubgroup}.
\end{proof}

\begin{convention}\label{conv-mainHCpair}
For the symplectic supermanifold $(\bb{R}^{2n\vert r},\omega_Q)$ with $Q$ having signature $(a,b)$, 
we set the following notation:

\begin{itemize}

\item $\mrm{Aut}_{2n\vert a,b}$ is the pro-supergroup of filtration preserving automorphisms of Poisson superalgebras 
$\mrm{Aut}_\mrm{Pois}^\mrm{filt}(\widehat{\mathcal{O}}_{\bb{R}^{2n\vert r},0},\omega_Q)$,

\item $\widehat{\mathcal{O}}_{2n\vert a,b}$ is the Poisson superalgebra $(\widehat{\mathcal{O}}_{\bb{R}^{2n\vert r},0},\{-,-\}_{\omega_Q})$, and

\item $\mfrk{g}_{2n\vert a,b}$ is the Lie superalgebra $\widehat{\mrm{Vect}}_0^\mrm{symp}(\bb{R}^{2n\vert r},\omega_Q)$.

\end{itemize}
\end{convention}

The pair $(\mfrk{g}_{2n\vert a,b},\mrm{Sp}(2n\vert a,b))$ will be the main sHC pair of interest to us. 

\subsection{Category of Manifolds}

In the Borel construction, 
one considers the category of principal bundles. 
Analogously, we will make use of a category of principal $(\mfrk{g},K)$-bundles.

Following \cite[Def. 1.5]{GGW}, we define principal bundles for Harish-Chandra pairs as follows. 

\begin{defn}
Let $\bb{X}$ be a supermanifold. Let $(\mfrk{g}, K)$ be a super-Harish-Chandra pair. 
A $(\mfrk{g},K)$-\emph{principal bundle} over $\bb{X}$ is 
a principal $K$-bundle $P\rta \bb{X}$ together with a $K$-invariant $\mfrk{g}$-valued 1-form~${\nu\in\Omega^1(P;\mfrk{g})}$ 
such that

\begin{itemize}

\item for all $a\in\mrm{Lie}(K)$, 
we have $\nu(\zeta_a)=a$ where $\zeta_a$ denotes the induced vector field on $P$, and 

\item $\nu$ satisfies the Maurer-Cartan equation

\[d_\mrm{dR}\omega+\frac{1}{2}[\nu,\nu]=0.\]

\end{itemize}

We let $\msf{Bun}_{(\mfrk{g},K)}^\mrm{flat}$ denote the category of $(\mfrk{g},K)$-principal bundles and morphisms $(P,\omega)\rta(P',\omega')$ bundle maps $F\colon P\rta P'$ so that $F^*\omega'=\omega$.
\end{defn}

\begin{ex}
For the sHC pair $({\mrm{Lie}(K)},K)$, the notion of a principal $({\mrm{Lie}(K)},K)$-bundle recovers the notion of a principal $K$-bundle with connection.
\end{ex}

\begin{ex}
If $\mfrk{g}$ is the Lie algebra of a Lie group $G$, 
and $K\subset G$ is a closed subgroup, 
then a principal $(\mfrk{g},K)$-bundle is the same as a principal $G$-bundle with connection, 
together with a reduction of structure group from $G$ to $K$. 
See \cite[Pg. 8]{GGW}.
\end{ex}

As noted in Convention \ref{conv-mainHCpair}, 
the sHC pair of interest to us is $\gK$. 
Let $(\bb{M},\omega)$ be a symplectic supermanifold of type $(2n\vert a,b)$.
Recall from Lemma \ref{lem-SymplecticFrame}, 
that there is a principal $\mrm{Sp}(2n\vert a,b)$-bundle $\mrm{Fr}^{\mrm{Sp}(2n\vert a,b)}_\bb{M}$ on $\bb{M}$. 
We will construct a $\gK$-principal bundle structure on $\mrm{Fr}^{\mrm{Sp}(2n\vert a,b)}_\bb{M}$. 
We will do this in two steps: 

\begin{itemize}

\item[\underline{Step 1.}] first we construct a principal $\mrm{Aut}_{2n\vert a,b}$-bundle 
$\bb{M}^\mrm{coor}$ on $\bb{M}$, 
and give it a $\mfrk{g}_{2n\vert a,b}$-valued connection; 

\item[\underline{Step 2.}] second we move this structure from~$\bb{M}^\mrm{coor}$ to $\mrm{Fr}^{\mrm{Sp}(2n\vert a,b)}_\bb{M}$ using a ``formal exponential."

\end{itemize}

The results of this section are summarized in the following theorem. 

\begin{thm}\label{thm-CategoryofMfldsGoal}
There is a category with objects symplectic supermanifolds $(\bb{M},\omega)$ of type $(2n\vert a,b)$ equipped with a $\gK$-bundle structure on $\mrm{Fr}^{\mrm{Sp}(2n\vert a,b)}_\bb{M}$, 
which we denote $\msf{sGK}^=_{2n\vert a,b}$. 
A choice of formal exponential defines a lift of an object in 
$\msf{sMfld}_{2n\vert 2,b}$ to an object of $\msf{sGK}^=_{2n\vert a,b}$. 

Moreover, there is a functor

\[\msf{sGK}^=_{2n\vert a,b}\rta \msf{Bun}^\mrm{flat}_{\gK}\]

living above $\msf{sMfld}_{2n\vert a,b}$. 
\end{thm}

\subsubsection{Formal Symplectic Coordinate Bundles}\label{subsubsec-FormalCoord}

Given a symplectic supermanifold $(\bb{M},\omega)$ of type $(2n\vert a,b)$, 
we will construct a $\mrm{Aut}_{2n\vert a,b}$-bundle $\bb{M}^\mrm{coor}$ on $\bb{M}$ 
with a $\mfrk{g}_{2n\vert a,b}$-valued connection. 
Since $\mrm{Aut}_{2n\vert a,b}$ is a pro-supergroup, $\bb{M}^\mrm{coor}$ will  be a pro-supermanifold. 

For $x\in\bb{M}$, by Lemma \ref{lem-PoissonSuperalgebra}, 
$\widehat{\mathcal{O}}_{\bb{M},x}$ has the structure of a Poisson superalgebra. 
In analogy with \cite[\S 2.2.1]{GGW}, 
we define the \emph{formal symplectic coordinate bundle} of $(\bb{M},\omega)$ 
to be the bundle with fiber over $x\in\bb{M}$ given by the group of isomorphisms of Poisson superalgebras,

\[\bb{M}^\mrm{coor}_x=\mrm{Isom}_{\mrm{Pois}}\left(\widehat{\mathcal{O}}_{\bb{M},x},\widehat{\mathcal{O}}_{2n\vert a,b}\right).\]

Since $\bb{M}$ is of type $(2n\vert a,b)$, we have an isomorphism of Poisson algebras 
$\widehat{\mathcal{O}}_{\bb{M},x}\rta\widehat{\mathcal{O}}_{2n\vert a,b}$, 
so that $\bb{M}^\mrm{coor}_x$ is nonempty. 
Moreover, this implies that $\bb{M}^\mrm{coor}_x$ is non-canonically isomorphic to 
\[\mrm{Isom}_{\mrm{Pois}}\left(\widehat{\mathcal{O}}_{2n\vert a,b},\widehat{\mathcal{O}}_{2n\vert a,b}\right)=\mrm{Aut}_{2n\vert a,b}.\]

\noindent See Convention \ref{conv-mainHCpair}.

As in \cite[\S 3.1]{BK}, one can construct $\bb{M}^\mrm{coor}$ by the functor $\mathcal{F}_{\bb{M}^\mrm{coor}}$ it represents. 
For $T$ another supermanifold with a map $\eta\colon T\rta\bb{M}$, 
let $T_\eta\subset T\times\bb{M}$ be the submanifold of pairs $(t,\eta(t))$. 
Let $\widehat{\mathcal{O}}_{\bb{M},\eta}$ be the ring of formal germs of functions on $T\times\bb{M}$ near $T_\eta$. 
The functor $\bb{M}^\mrm{coor}$ represents sends an affine 
space $T$ to the set of pairs 
\[\mathcal{F}_{\bb{M}^\mrm{coor}}(T)=\{(\eta,\phi:\eta\colon T\rta\bb{M}, \phi\colon\widehat{\mathcal{O}}_{\bb{M},\eta}\cong \mathcal{O}_T\hat{\otimes}\AQ\}.\] 

Define an action of $\mrm{Aut}_{2n\vert a,b}$ on $\bb{M}^\mrm{coor}_x$ by post-composition. 
Under this action, $\bb{M}^\mrm{coor}\rta\bb{M}$ becomes a principal $\mrm{Aut}_{2n\vert a,b}$-bundle. 

Analogous to how a smooth map induces a map on frame bundles, 
we have the following.

\begin{lem}\label{lem-CoorFunctorial}
Let $f\colon \bb{M}_1\rta\bb{M}_2$ be a morphism in $\msf{sMfld}_{2n\vert a,b}$. 
There is a morphism 

\[f^\mrm{coor}\colon\bb{M}_1^\mrm{coor}\rta\bb{M}_2^\mrm{coor}\]

of $\mrm{Aut}_{2n\vert a,b}$-bundles.
\end{lem}

\begin{proof}
We produce a map 
$f^\mrm{coor}\colon\bb{M}_1^\mrm{coor}\rta\bb{M}_2^\mrm{coor}$ 
 from a natural transformation between the functors 
$\mathcal{F}_{\bb{M}_1^\mrm{coor}}$ and $\mathcal{F}_{\bb{M}_2^\mrm{coor}}$ 
that these spaces represent. 
Fix a supermanifold $T$. 
Let $(\eta,\phi)\in \mathcal{F}_{\bb{M}_1^\mrm{coor}}(T)$. 
Set $\eta'$ to be the composite 

\[T\xrta{\eta} \bb{M}_1\xrta{f}\bb{M}_2.\] 

Then $f^*\colon\mathcal{O}_{\bb{M}_2}\rta\mathcal{O}_{\bb{M}_1}$ descends to an equivalence on completions 
\[\bar{f}^*\colon \widehat{\mathcal{O}}_{\bb{M}_2,\eta'}\rta \widehat{\mathcal{O}}_{\bb{M}_1,\eta}\] 
since $(\mrm{Id}_T\times f)(T_\eta)=T_{\eta'}$. 
Thus, $(\eta',\phi\circ \bar{f}^*)$ defines an element of $\mathcal{F}_{\bb{M}_2^\mrm{coor}}(T)$. 
This assignment on objects extends to a natural transformation, 
resulting in a morphism on representing spaces 
$f^\mrm{coor}\colon\bb{M}_1^\mrm{coor}\rta\bb{M}_2^\mrm{coor}$. 
\end{proof}

As in \cite[Lem. 3.2]{BK} and \cite[Def. 2.2]{GGW}, we have the following.

\begin{lem}\label{lem-McoorConn}
Let $(\bb{M},\omega)$ be a symplectic supermanifold of type $(2n\vert a,b)$. 
There is a connection 1-form 
$\nu^\mrm{coor}\in\Omega^1(\bb{M}^\mrm{coor};\mfrk{g}_{2n\vert a,b})$.
\end{lem}

\begin{proof}
The principal $\mrm{Aut}_{2n\vert a,b}$-bundle 
$\pi\colon\bb{M}^\mrm{coor}\rta\bb{M}$ 
determines a short exact sequence of pro-vector bundles on the pro-supermanifold $\bb{M}^\mrm{coor}$, 

\[0\rta\mrm{ker}(d\pi)\rta T\bb{M}^\mrm{coor}\rta T\bb{M}\rta 0.\]

The kernel $\mrm{ker}(d\pi)$ is isomorphic to the trivial $\mrm{Lie}(\mrm{Aut}_{2n\vert a,b})$-bundle on $\bb{M}^\mrm{coor}$. 
Hence, at each point $(x,\varphi)\in\bb{M}^\mrm{coor}$, 
we get a short exact sequence

\[0\rta \mrm{Lie}(\mrm{Aut}_{2n\vert a,b})\rta T_{(x,\varphi)}\bb{M}^\mrm{coor}\rta T_x\bb{M}\rta 0.\]

The isomorphism $\varphi$ determines an equivalence $T_x\bb{M}\cong\bb{R}^{2n\vert r}$. 
This gives us an equivalence 

\[\mrm{Lie}(\mrm{Aut}_{2n\vert a,b})\oplus \bb{R}^{2n\vert r}\simeq T_{(x,\varphi)}\bb{M}^\mrm{coor}.\]

Using the symplectic analogue of (\ref{eq-VanishAtZero}), 
we get an equivalence 

\[\nu^\mrm{coor}_{x,\varphi}\colon T_{(x,\varphi)}\bb{M}^\mrm{coor}\xrta{\sim} \mfrk{g}_{2n\vert a,b}.\]

Sending a point $(x,\varphi)\in\bb{M}^\mrm{coor}$ to the map $\nu^\mrm{coor}_{x,\varphi}$ defines a one-form 
\[\nu^\mrm{coor}\in\Omega^1(\bb{M}^\mrm{coor},\mfrk{g}_{2n\vert a,b}).\qedhere\]
\end{proof}

\begin{cor}
The connection $\nu^\mrm{coor}$ is flat. 
\end{cor}

\begin{proof}
By construction, $\nu^\mrm{coor}$ is the inverse of a Lie superalgebra map 
\[\mfrk{g}_{2n\vert a,b}\rta\mrm{Vect}(\bb{M}^\mrm{coor}),\]
and hence satisfies the Maurer-Cartan equation.
\end{proof}

This completes Step 1 of our proof of Theorem \ref{thm-CategoryofMfldsGoal}.

\begin{cor}\label{cor-CoorStructure} 
The formal symplectic coordinate bundle $\bb{M}^\mrm{coor}\rta\bb{M}$ is a 
$(\mfrk{g}_{2n\vert a,b},\mrm{Aut}_{2n\vert a,b})$-principal bundle.
\end{cor}

One can compare this to \cite[Lem. 3.2]{BK} for the non-super version. 

\subsubsection{Formal Exponentials}

We define the notion of a ``formal exponential" which will allow us to move the structure defined in Corollary \ref{cor-CoorStructure} 
from the formal symplectic coordinate bundle, to the symplectic frame bundle. 

As in \cite[Def. 2.4]{GGW}, we define a formal exponential as follows. 

\begin{defn}\label{def-FormalExponential}
Let $(\bb{M},\omega)$ be a symplectic supermanifold of type $(2n\vert a,b)$. 
A \emph{formal exponential} on $\bb{M}$ is a section of 
the $\mrm{Aut}_{2n\vert a,b}/\mrm{Sp}(2n\vert a,b)$-bundle 

\[\mrm{Exp}(\bb{M}):=\bb{M}^\mrm{coor}/\mrm{Sp}(2n\vert a,b)\rta\bb{M}.\]
\end{defn}

\begin{lem}\label{lem-FiberContractible}
The space $\mrm{Aut}_{2n\vert a,b}/\mrm{Sp}(2n\vert a,b)$ is contractible, 
and thus formal exponentials always exist.
\end{lem}

\begin{proof}
Note that $\widehat{\mathcal{O}}_{2n\vert a,b}/\mfrk{m}^2$ consists of linear functions on $\bb{R}^{2n\vert a+b}$. 
The image of $\mrm{Aut}_{2n\vert a,b}$ in 
$\mrm{Aut}_\mrm{Pois}(\widehat{\mathcal{O}}_{2n\vert a,b}/\mfrk{m}^2)$ 
is thus $\mrm{Sp}(2n\vert a,b)$. 
Let $\mrm{Aut}^+_{2n\vert a,b}$ be the kernel of the projection 

\[1\rta\mrm{Aut}^+_{2n\vert a,b}\rta\mrm{Aut}_{2n\vert a,b}\rta\mrm{Sp}(2n\vert a,b)\rta 1.\]
Since $\mrm{Aut}^+_{2n\vert a,b}$ is pro-nilpotent, it is contractible. 
In fact, $\mrm{Aut}^+_{2n\vert a,b}$ is a pro-vector space.
The result follows.
\end{proof} 

We will use the following to define a morphism between formal exponentials. 
Let $f\colon \bb{M}_1\rta\bb{M}_2$ be a morphism in $\msf{sMfld}_{2n\vert a,b}$. 
The map $f^\mrm{coor}$ from Lemma \ref{lem-CoorFunctorial} respects the action of $\mrm{Aut}_{2n\vert a,b}$ on both sides by post-composition, 
and therefore descends to a map 
$\mrm{Exp}(\bb{M}_1)\rta\mrm{Exp}(\bb{M}_2)$ 
of pro-supermanifolds, and a commuting diagram 

\[\begin{xymatrix}
{
\mrm{Exp}(\bb{M}_1)\arw[d]\arw[r]^{f^\mrm{coor}}& \mrm{Exp}(\bb{M}_2)\arw[d]\\
\bb{M}_1\arw[r]_f & \bb{M}_2.
}
\end{xymatrix}\]

We obtain the following diagram involving the pullback bundle $f^*\mrm{Exp}(\bb{M}_2)$, 

\[\begin{xymatrix}
{
\mrm{Exp}(\bb{M}_1)\arw[dr]^g\ar@/^15pt/[rrd]\ar@/_15pt/[ddr] & & \\
& f^*\mrm{Exp}(\bb{M}_2)\arw[r]^{f^\mrm{coor}}\arw[d] & \mrm{Exp}(\bb{M}_2)\arw[d]\\
& \bb{M}_1\arw[r]_f & \bb{M}_2.
}
\end{xymatrix}\]

Given a formal exponential $\sigma_1$ on $\bb{M}_1$, 
we obtain a section of $f^\mrm{Exp}(\bb{M}_2)$ by $g\circ\sigma_1$. 
On the other hand, given a formal exponential $\sigma_2$ on $\bb{M}_2$, 
we get a section of $f^*\mrm{Exp}(\bb{M}_2)$ from the diagram 

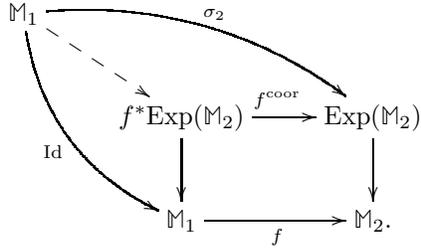
\[\begin{xymatrix}
{
\bb{M}_1\ar@{-->}[dr]\ar@/^15pt/[rrd]^{\sigma_2}\ar@/_15pt/[ddr]_{\mrm{Id}} & & \\
& f^*\mrm{Exp}(\bb{M}_2)\arw[r]^{f^\mrm{coor}}\arw[d] & \mrm{Exp}(\bb{M}_2)\arw[d]\\
& \bb{M}_1\arw[r]_f & \bb{M}_2.
}
\end{xymatrix}\]

\begin{defn}\label{def-sGK}
Let $\msf{sGK}_{2n\vert a,b}$ denote the category with

\begin{itemize}

\item objects are pairs $((\bb{M},\omega),\sigma)$ where 
$(\bb{M},\omega)$ is a symplectic supermanifold of type $(2n\vert a,b)$ and $\sigma$ is a formal exponential on $\bb{M}$, and 

\item a morphism $(\bb{M}_1,\sigma_1)\rta(\bb{M}_2,\sigma_2)$ is 
a morphism $f\colon\bb{M}_1\rta\bb{M}_2$ in $\msf{sMfld}_{2n\vert a,b}$, 
and a homotopy class of paths in the space 
$\Gamma(\bb{M}_1,f^*\mrm{Exp}(\bb{M}_2))$ 
between the sections defined by $\sigma_1$ and $\sigma_2$.
\end{itemize}
\end{defn}

One should compare the following with the discussion around \cite[Def. 2.11]{GGW}. 
\begin{lem}
The forgetful functor 
$\msf{sGK}_{2n\vert a,b}\rta\msf{sMfld}_{2n\vert a,b}$ 
is an equivalence of categories. 
\end{lem}

\begin{proof}
By our definition of morphism spaces in $\msf{sGK}_{2n\vert a,b}$, 
it suffices to show that for $f\colon\bb{M}_1\rta\bb{M}_2$ a morphism in $\msf{sMfld}_{2n\vert a,b}$, 
the space $\Gamma(\bb{M}_1,f^*\mrm{Exp}(\bb{M}_2))$ is contractible. 
The bundle $\mrm{Exp}(\bb{M}_2)$, and hence $f^*\mrm{Exp}(\bb{M}_2)$, 
have fiber $\mrm{Aut}_{2n\vert a,b}/\mrm{Sp}(2n\vert a,b)$. 
By Lemma \ref{lem-FiberContractible}, 
this fiber is contractible. 
The space of sections of a bundle with contractible fiber is contractible. 
This completes the proof.
\end{proof}

We introduce a stricter variation on $\msf{sGK}_{2n\vert a,b}$.

\begin{variant}\label{var-sgk=}
Let $\msf{sGK}_{2n\vert a,b}^=$ be the category with 
\begin{itemize}

\item objects: pairs $((\bb{M},\omega),\sigma)$ where 
$(\bb{M},\omega)$ is a symplectic supermanifold of type $(2n\vert a,b)$ and $\sigma$ is a formal exponential on $\bb{M}$, and 

\item morphisms: a map $(\bb{M}_1,\sigma_1)\rta(\bb{M}_2,\sigma_2)$ is 
a morphism $f\colon\bb{M}_1\rta\bb{M}_2$ in $\msf{sMfld}_{2n\vert a,b}$ 
such that the diagram 

\[\begin{xymatrix}
{
\mrm{Exp}(\bb{M}_1)\arw[r]^{f^\mrm{coor}} & \mrm{Exp}(\bb{M}_2)\\
\bb{M}_1\arw[u]^{\sigma_1}\arw[r]_f & \bb{M}_2\arw[u]_{\sigma_2}
}
\end{xymatrix}\]

commutes.
\end{itemize}

\end{variant}

\noindent Note that the condition on morphisms is equivalent to asking that 
$f\colon\bb{M}_1\rta\bb{M}_2$ be such that 
the sections in $\Gamma(\bb{M}_1,f^*\mrm{Exp}(\bb{M}_2))$ 
defined by $\sigma_1$ and $\sigma_2$ are equal, 
as opposed to having a path between them.

\begin{rmk}
There is an evident functor 
$\msf{sGK}^=_{2n\vert a,b}\rta\msf{sGK}_{2n\vert a,b}$ 
that is the identity on objects and the inclusion of constant paths on morphisms, 
but this functor is not fully faithful.
\end{rmk}

The main use of a formal exponential is to put the structure of a $\gK$-bundle structure on the symplectic frame bundle. 
The following is analogous to \cite[Prop. 2.6]{GGW}.

\begin{prop}\label{prop-gkfromExp}
Let $\sigma\in\Gamma(\bb{M},\mrm{Exp}(\bb{M}))$ be a formal exponential. 
Then 
\begin{itemize}
\item $\sigma$ lifts to a $\mrm{Sp}(2n\vert a,b)$-equivariant map $\tilde{\sigma}\colon\Fr_\bb{M}\rta\bb{M}^\mrm{coor}$,
\[\xymatrix{
\Fr_\bb{M}\arw[r]^{\tilde{\sigma}}\arw[d] & \bb{M}^\mrm{coor}\arw[d]\\
\bb{M}\arw[r]^-\sigma & \mrm{Exp}(\bb{M}),
}\]
\item and $\tilde{\sigma}^*(\nu^\mrm{coor})$ is a flat $\mfrk{g}_{2n\vert a,b}$-valued connection on $\Fr_\bb{M}$. 
With this connection, $\Fr_\bb{M}$ is a $\gK$-bundle.
\item Any two choices of a formal exponential determine 
$\gK$-bundle structures on $\Fr_\bb{M}$ that are gauge-equivalent. 
\end{itemize}
\end{prop}

\begin{proof}
The first claim follows from the the definition of 
$\mrm{Exp}(\bb{M})$ as $\bb{M}^\mrm{coor}/\mrm{Sp}(2n\vert a,b)$ 
and the equivalence 
$\bb{M}\simeq \Fr_\bb{M}/\mrm{Sp}(2n\vert a,b)$. 
The second claim is just the statement that flat connections pullback, but in the context of pro-supermanifolds. 
The gauge-equivalence in the third claim can be produced using the contractibility of the space of formal exponentials.
\end{proof}

\begin{rmk}\label{rmk-BKours2}
In the language of \cite[\S 2.4]{BK}, 
Proposition \ref{prop-gkfromExp} shows that 
one can think of a formal exponential on $\bb{M}$ as determining a ``lift" 
(similar to a reduction of structure group) of the 
$(\mfrk{g}_{2n\vert a,b},\mrm{Aut}_{2n\vert a,b})$-bundle 
$(\bb{M}^\mrm{coor},\nu^\mrm{coor})$ from Corollary \ref{cor-CoorStructure} 
to the sHC pair $\gK$. 
Suppose $\bb{M}$ is purely even so that $M=\bb{M}$ is an ordinary symplectic manifold. 
In \cite[Lem. 3.4]{BK}, 
Bezrukavnikov-Kaledin describe an HC pair they call 
$(\mrm{Der}(D),\mrm{Aut}(D))$ 
(where $D$ is the Weyl algebra) and 
show that the set of lifts of $M^\mrm{coor}$ to a 
$(\mrm{Der}(D),\mrm{Aut}(D))$-bundle is in bijective correspondence with 
isomorphism classes of deformation quantizations $Q(M,\omega)$. 
See Remark \ref{rmk-BKours3} below and Remark \ref{rmk-BKours} above for further discussion in this direction. 
\end{rmk}

\begin{cor}\label{cor-FrFunctor}
There is a functor 
\[\mrm{Fr}\colon\msf{sGK}^=_{2n\vert a,b}\rta\msf{Bun}^\mrm{flat}_{\gK}\]
sending a pair $((\bb{M},\omega),\sigma)$ to the symplectic frame bundle $\Fr_\bb{M}$ with $\gK$-bundle structure induced from $\sigma$.
\end{cor}

\begin{proof}
This functor is defined on objects by Proposition \ref{prop-gkfromExp}. 
We need to define $\mrm{Fr}$ on morphisms. 
Let $(\bb{M}_1,\sigma_1)$ and $(\bb{M}_2,\sigma_2)$ be two objects in $\msf{sGK}_{2n\vert a,b}^=$, 
and $f\colon\bb{M}_1\rta\bb{M}_2$ a morphism between them. 
The map $f$ induces a map of $\mrm{Sp}(2n\vert a,b)$-bundles 

\[df\colon\Fr_{\bb{M}_1}\rta\Fr_{\bb{M}_2}.\] 

For $df$ to be a map of $\gK$-bundles, 
we need the $\mfrk{g}_{2n\vert a,b}$-valued connection on $\Fr_{\bb{M}_2}$ to pullback to the one on $\Fr_{\bb{M}_1}$. 
By definition of the functor $\mrm{Fr}$ on objects, 
the connection on $\Fr_{\bb{M}_i}$ is 
$\tilde{\sigma}_i^*(\nu^\mrm{coor}_{\bb{M}_i})$, 
for $i=1,2$. 
Thus it suffices to show that there is an equality 
\[(df)^*\tilde{\sigma}_2^*(\nu^\mrm{coor}_{\bb{M}_2})=\tilde{\sigma}_1^*(\nu^\mrm{coor}_{\bb{M}_1}).\] 
This follows from the commutativity of the cube 

\begin{center}
\begin{tikzcd}[row sep=scriptsize, column sep=scriptsize]
& \Fr_{\bb{M}_1} \arrow[dl,"df" description] \arrow[rr,"\tilde{\sigma}_1" description] \arrow[dd] & & \bb{M}_1^\mrm{coor} \arrow[dl,"f^\mrm{coor}" description] \arrow[dd] \\
\Fr_{\bb{M}_2} \arrow[rr, crossing over,"\tilde{\sigma}_2" description] \arrow[dd] & & \bb{M}_2^\mrm{coor} \\
& \bb{M}_1 \arrow[dl,"f" description] \arrow[rr,"\sigma_1" description] & & \mrm{Exp}(\bb{M}_1) \arrow[dl,"f^\mrm{coor}" description] \\
\bb{M}_2 \arrow[rr,"\sigma_2" description] & & \mrm{Exp}(\bb{M}_2) \arrow[from=uu, crossing over]\\
\end{tikzcd}.
\end{center}
\end{proof}

\subsubsection{Cotangent Bundle Example}

We discuss formal exponentials on symplectic supermanifolds of the form in Example \ref{ex-Rothstein}. 
Recall from Theorem \ref{thm-Rothstein} that symplectic supermanifolds non-canonically look like $(E[1],\tilde{\omega})$ where 
$E\rta M$ is a vector bundle on an ordinary symplectic manifold $(M,\omega)$, 
and $\tilde{\omega}$ is defined using a metric $g$ on $E$ and a compatible connection $\nabla$.

Recall that a \emph{symplectic connection} on an ordinary symplectic manifold $(M,\omega)$ is a torsion-free connection so that $\omega$ is constant with respect to the covariant derivative, 
see for example \cite[Def. 2.1]{BCG} or \cite[Def. 2.1]{Gelfand}.

\begin{lem}\label{lem-ExpConnection}
If $\bb{M}=E[1]$ is the symplectic supermanifold defined in Example \ref{ex-Rothstein} from the data $(M,\omega,E,g,\nabla)$, 
then a symplectic connection on $M$ determines a formal exponential on $\bb{M}$.
\end{lem}

See \cite[Pg. 3-4]{Engeli} for a description of the resulting differential on $\mathbf{desc}_{(\bb{M},\sigma)}(-)$. 

\begin{proof}
Willwacher in \cite[\S 2.5]{Willwacher} has shown that a torsion-free connection on an ordinary manifold $X$ gives a section of $X^\mrm{coor}$. 
So a torsion-free connection produces a compatible choice of, for each $x\in X$, an isomorphism 
\[\bb{k}[[x_1,\dots,x_n]]=\widehat{\mathcal{O}}_{n}\simeq\widehat{\mathcal{O}}_{X,x}.\]

Similarly, a symplectic connection on an ordinary symplectic manifold $M$ gives a section of $M^\mrm{coor}$. 
So a symplectic connection produces a compatible choice of, for each $x\in M$, an isomorphism of Poisson algebras 
$\widehat{\mathcal{O}}_{2n}\simeq\widehat{\mathcal{O}}_{M,x}$. 

In the purely odd case, a connection on a vector bundle $E\rta X$ produces 
a compatible choice of, for each $x\in X$, an isomorphism of algebras 
\[\Lambda^\bullet[[\theta_1,\dots,\theta_r]]\simeq\Gamma(X,\Lambda^\bullet E)\hat{}_x.\]

Combining these, a symplectic connection on $M$ and a metric connection on a quadratic vector bundle $E\rta M$ produces a compatible choice of, 
for each $x\in X$, 
an isomorphism of Poisson super-algebras 
\[\widehat{\mathcal{O}}_{2n\vert a,b}\simeq(\widehat{\mathcal{O}}_{E[1]})_x\simeq\Gamma(M,\Lambda^\bullet E)\hat{}_x.\]
This data is a formal exponential on $E[1]$.
\end{proof}

\begin{rmk}\label{rmk-VBsGK}
Given an ordinary manifold $X$, the symplectic manifold $T^*X$ has a canonical symplectic connection. 

Let $\pi\colon T^*X\rta X$ be the projection. 
Consider the functor 

\[T^*\colon \msf{VB}^{\mrm{quad},\nabla}_{/X}\rta\msf{VB}^{\mrm{quad},\nabla}_{/T^*X}\]

sending a $E\rta X$ to $\pi^*E\rta T^*X$ and the metric and connection on $E$ to the pullback metric and connection, respectively. 
Using the canonical symplectic connection on the cotangent bundle \cite{BNW}, 
Lemma \ref{lem-ExpConnection} allows one to define a lift 

\[\begin{xymatrix}
{
& & \msf{sGK}\arw[d]\\
\msf{VB}_{/X}^{\mrm{quad},\nabla}\arw[r]_{T^*}\ar@{-->}[urr]^{\tilde{L}} & \msf{VB}^{\mrm{quad},\nabla}_{/T^*X}\arw[r] & \msf{sMfld}^\mrm{Sp}
}
\end{xymatrix}\]

\noindent where the categories of manifolds here are not restricted to a particular type $(2n\vert a,b)$. 

Just as the cotangent bundle $T^*X$ has a canonical deformation quantization by differential operators on $X$, 
the lift $\tilde{L}$ will allow us to construct deformation quantizations for symplectic supermanifolds built from vector bundles over the cotangent bundle.
\end{rmk}

\subsection{Descent Functor}\label{subsec-Descent}

We will discuss Harish-Chandra descent for the sHC pair $\gK$. 
After studying some monoidal properties of this descent functor, 
we will construct the super-Gelfand-Kazhdan descent functor that will be used in later sections. 

\begin{convention}\label{conv-HC}
Throughout this section, let 

\begin{itemize}

\item $(\mfrk{g},K)$ be an sHC pair, 

\item $\msf{Mod}_{(\mfrk{g},K)}$ denote the category of $(\mfrk{g},K)$-modules, 

\item $\msf{Mod}^\mrm{fin}_{(\mfrk{g},K)}$ denote the category of finite-dimensional $(\mfrk{g},K)$-modules, 

\item $\msf{VB}^\mrm{flat}_{2n\vert a,b}$ denote the category, 
fibered over $\msf{sMfld}_{2n\vert a,b}$, 
of flat finite-dimensional vector bundles, 

\item $\msf{VB}^\mrm{flat}_{/\bb{M}}$ denote the category of flat finite-dimensional vector bundles over a symplectic supermanifold $\bb{M}$, 

\item $\msf{Pro}(\msf{VB}_{/\bb{M}})^\mrm{flat}$ denote the category of pro-objects in $\msf{VB}_{/\bb{M}}$ together with a flat connection, and  

\item $\msf{Mod}_{\Omega^\bullet}$ denote the category, 
fibered over $\msf{sMfld}_{2n\vert a,b}$, 
of symplectic supermanifolds $(\bb{M},\omega)$ together with a module over the superalgebra $\Omega^\bullet_\bb{M}$. 
\end{itemize}
\end{convention}

Given a flat $(\mfrk{g}, K)$-bundle $P\rta \bb{M}$ with connection 1-form $\nu\in\Omega^1(P;\mfrk{g})$ and a finite-dimensional $(\mfrk{g},K)$-module $V$, 
we obtain a vector bundle on $\bb{M}$ using the Borel construction, 
$P\times_KV$. 
We can equip $P\times_KV$ with a flat connection using $\omega$ and the action $\rho_\mfrk{g}^V$ of $\mfrk{g}$ on $V$ as follows. 
The action of $\mfrk{g}$ on $V$ induces a map

\begin{align}\label{eq-Rho}
\rho_\mfrk{g}^V(\nu)\colon \Omega^\bullet(P;\underline{V})\rta\Omega^{\bullet+1}(P;\underline{V})
\end{align}

defined by $\rho_\mfrk{g}^V(\nu)(-)=\rho_\mfrk{g}^V(\nu\wedge -)$. 
Now, $\nabla^{P,V}=d_\mrm{dR, P}+\rho_\mfrk{g}^V(\nu)$ defines a differential on the subalgebra of basic forms, 
and hence a flat connection on $P\times_KV$. 
See \cite[Lem. 1.12]{GGW} for the non-super case.

As in \cite[Def. 1.14]{GGW}, given an sHC-pair $(\mfrk{g},K)$, \emph{Harish-Chandra} descent is the resulting functor

\[\mrm{desc}\colon\msf{Bun}_{(\mfrk{g},K)}^\mrm{flat}\times \msf{Mod}^\mrm{fin}_{(\mfrk{g},K)}\rta\msf{VB}^\mrm{flat}_{2n\vert a,b}\] 

\noindent sending $(P\rta\bb{M}, V)$ to $(P\times_KV\rta \bb{M},\nabla^{P,V})$.

Taking the de Rham complex of the flat vector bundle produces a functor 

\[\mathbf{desc}\colon(\msf{Bun}_{(\mfrk{g},K)}^\mrm{flat})^\mrm{op}\times\msf{Mod}^\mrm{fin}_{(\mfrk{g},K)}\rta\msf{Mod}_{\Omega^\bullet}.\]

\begin{ex}
Take $(\mfrk{g},K)$ to be the sHC pair $\gK$. 
Restricting along the functor $\mrm{Fr}$ of Corollary \ref{cor-FrFunctor}, 
we obtain a descent functor 
\[\msf{sGK}^=_{2n\vert a,b}\times\msf{Mod}^\mrm{fin}_{\gK}\rta\msf{VB}^\mrm{flat}_{2n\vert a,b}.\]
\end{ex}

\subsubsection{Monoidal Properties of Descent}

Let $(\mfrk{g},K)$ be an sHC pair. 
Restricting to a fixed $(\mfrk{g},K)$-bundle $(P\rta\bb{M},\nu)$, we have a functor 

\[\mrm{desc}_{P,\nu}\colon \msf{Mod}^\mrm{fin}_{(\mfrk{g},K)}\rta \msf{VB}^\mrm{flat}_{/\bb{M}}.\]

The category $\msf{Mod}^\mrm{fin}_{(\mfrk{g},K)}$ has a symmetric monoidal structure given by $\otimes_\bb{k} $. 
The category $\msf{VB}^\mrm{flat}_{/\bb{M}}$ has a symmetric monoidal structure by taking tensor product of vector bundles and flat connections. 

The following foundational observation allows us to deduce several nice properties of Harish-Chandra, 
and in particular super-Gelfand-Kazhdan, descent.

\begin{prop}
The functor $\mrm{desc}_{P,\nu}$ is symmetric monoidal.
\end{prop}

\begin{proof}
Let $V,W\in\msf{Mod}^\mrm{fin}_{(\mfrk{g},K)}$. 
The Borel construction is symmetric monoidal, 
\[P\times_K(V\otimes W)\simeq (P\times_KV)\otimes (P\times_KW),\]
as one can check on fibers. 
It therefore suffices to show that the connection on $\mrm{desc}_{P,\nu}(V\otimes W)$ is the tensor product of the connection on $\mrm{desc}_{(P,\nu)}(V)$ and on $\mrm{desc}_{(P,\nu)}(W)$. 
From the construction of the connection, Equation (\ref{eq-Rho}) or \cite[\S 1.3.2]{GGW}, 
we have $\nabla^{P,V\otimes W}=d_{\mrm{dR}, P}+\rho_{\mfrk{g}}^{V\otimes W}$, 
where $\rho_\mfrk{g}^{V\otimes W}(\nu)$ is defined using the action of $\mfrk{g}$ on $V\otimes W$. 
Since the tensor product $V\otimes W$ is taken in $\msf{Mod}_{(\mfrk{g},K)}$, 
we have $\rho_\mfrk{g}^{V\otimes W}(\nu)=\rho_\mfrk{g}^V(\nu)\otimes\rho_\mfrk{g}^W(\nu)$.
\end{proof}

From this proposition, we will be able to deduce several corollaries of how the descent functors interact with algebraic structures. 

\begin{cor}
The de Rham complex functor $\mathbf{desc}_{(P,\nu)}\colon\msf{Mod}^\mrm{fin}_{(\mfrk{g},K)}\rta\msf{Mod}_{\Omega^\bullet_\bb{M}}$ is symmetric monoidal.
\end{cor}

\begin{proof}
The functor $\msf{VB}^\mrm{flat}_{/\bb{M}}\rta\msf{Mod}_{\Omega^\bullet_\bb{M}}$ is symmetric monoidal.  
\end{proof}

\begin{cor}
Let $\bb{k} \in\msf{Mod}_{(\mfrk{g},K)}$ be the unit module. 
Then $\mrm{desc}_{(P,\nu)}(\bb{k} )$ is the trivial line bundle on $\bb{M}$ with connection given by the de Rham differential  
and $\mathbf{desc}_{(P,\nu)}(\bb{k} )$ is $\Omega^\bullet_\bb{M}$. 
\end{cor}

\begin{proof}
Symmetric monoidal functors take units to units. 
The units of $\msf{VB}^\mrm{flat}_{/\bb{M}}$ and $\msf{Mod}_{\Omega^\bullet_\bb{M}}$ are as described. 
\end{proof}

\begin{ex}
In particular, the space of horizontal sections of $\mrm{desc}_{(P,\nu)}(\bb{k} )$ is $\mathcal{O}_\bb{M}$.
\end{ex}

For a symmetric monoidal $\bb{k} $-linear category $\mathcal{V}$, 
let $\msf{Alg}(\mathcal{V})$ denote the category of algebra 
objects in $\mathcal{V}$. 

\begin{cor}
The descent functors lifts to symmetric monoidal functors on the level of algebra objects, 

\[\msf{Alg}(\msf{Mod}^\mrm{fin}_{(\mfrk{g},K)})\rta \msf{Alg}(\msf{VB}^\mrm{flat}_{/\bb{M}})\]

and

\[\msf{Alg}(\msf{Mod}^\mrm{fin}_{(\mfrk{g},K)})\rta \msf{Alg}(\msf{Mod}_{\Omega^\bullet_\bb{M}}).\]
\end{cor}

Note that an algebra object in $\Omega^\bullet_\bb{M}$-modules is just a $\Omega^\bullet_\bb{M}$-algebra. 

\begin{proof}
Symmetric monoidal functors induce symmetric monoidal functors on categories of algebra objects. 
\end{proof}

\begin{ex}
Take $(\mfrk{g},K)$ to be the sHC pair $\gK$. 
Then $\widehat{\mathcal{O}}_{2n\vert a,b}$ is an object in 
\[\msf{Alg}(\msf{Mod}_{\gK}),\] 
but the underlying $\gK$-module of $\widehat{\mathcal{O}}_{2n\vert a,b}$ is not finite-dimensional. 
However, $\widehat{\mathcal{O}}_{2n\vert a,b}$ is a limit of finite-dimensional modules, Definition \ref{def-formalnbhd}.
\end{ex}

To include the above example, we extend the descent functors to pro-objects. 
For $\mathcal{C}$ a category, let $\msf{Pro}(\mathcal{C})$ denote the category of pro-objects in $\mathcal{C}$. 
Note that if $\mathcal{C}$ is a symmetric monoidal category, then so is Pro(C), 
with tensor product given levelwise, see \cite[\S 4.2]{Davis-Lawson}.  
Since $\mrm{Pro}(-)$ is a functor between categories of categories, 
we obtain functors 

\[\msf{Pro}\left(\msf{Alg}(\msf{Mod}^\mrm{fin}_{(\mfrk{g},K)})\right)\rta \msf{Pro}\left(\msf{Alg}(\msf{VB}^\mrm{flat}_{/\bb{M}})\right)\]

and

\[\msf{Pro}\left(\msf{Alg}(\msf{Mod}^\mrm{fin}_{(\mfrk{g},K)})\right)\rta \msf{Pro}\left(\msf{Alg}_{\Omega^\bullet_\bb{M}}\right).\]

By definition, 
$\widehat{\mathcal{O}}_{2n\vert a,b}=\lim_i \mathcal{O}_{\bb{R}^{2n\vert r},0}/\mfrk{m}_0$ 
is a pro-object in algebras in $\msf{Mod}^\mrm{fin}_{\gK}$. 

\begin{ex}[Jet Bundles]
Given a vector bundle $E\rta\bb{M}$, 
the \emph{infinite jet bundle} $J^\infty(E)$ is a pro-object in $\msf{VB}_{/\bb{M}}$, see \cite[\S A.2]{GGJets}. 
Moreover, given a flat connection on $E$, 
$J^\infty(E)$ has a canonical flat connection \cite[Prop. A.8]{GGJets} so that 
$J^\infty(E)\in\msf{Pro}\left(\msf{VB}_{/\bb{M}}\right)^\mrm{flat}$. 
See also \cite[\S 2]{CFT}. 
\end{ex}

\begin{lem}\label{lem-DescO}
Let $((\bb{M},\omega),\sigma)\in\msf{sGK}^=_{2n\vert a,b}$. 
Then descending $\widehat{\mathcal{O}}_{2n\vert a,b}$ along $\Fr_\bb{M}$ 
produces the jet bundle of the trivial line bundle $\underline{k}_\bb{M}$ with its canonical flat connection,  

\[\mrm{desc}_{(\Fr_\bb{M},\sigma)}(\widehat{\mathcal{O}}_{2n\vert a,b})=J^\infty(\underline{\bb{k} }_\bb{M})\]

and thus 

\[\mathbf{desc}_{(\Fr_\bb{M},\sigma)}(\widehat{\mathcal{O}}_{2n\vert a,b})=\Omega^\bullet_\bb{M}.\]
\end{lem}
\noindent In particular, using \cite[Prop. A.8]{GGJets}, taking zero sections we see that $\widehat{\mathcal{O}}_{2n\vert a,b}$ descends to $\mathcal{O}_{\bb{M}}$.

One should compare this Lemma to \cite[Prop. 2.20]{GGW} or \cite[Pg. 20]{BK}.

\begin{proof}
The second claim follows from the first, so it suffices to produce an isomorphism of flat pro-bundles.

The bundle obtained by descending $\widehat{\mathcal{O}}_{2n\vert a,b}$ is 

\[\mrm{desc}_{(\Fr_\bb{M},\sigma)}(\widehat{\mathcal{O}}_{2n\vert a,b})=\Fr_\bb{M}\times_{\mrm{Sp}(2n\vert a,b)}\widehat{\mathcal{O}}_{2n\vert a,b}.\]

A point in the right-hand side is an equivalence class of 
a point $(x,\phi)$ in the frame bundle and a function $\hat{f}$ on the formal disk. 
The frame $\phi$ determines an isomorphism between a neighborhood $U_x$ of $x$ in $\bb{M}$ and the space $\bb{R}^{2n\vert a+b}$. 
Composing $\hat{f}$ and $\phi$, we obtain a germ of a function on $U_x$ at $x$; 
that is, an element $\hat{f}_\phi$ of the completion $(\widehat{\mathcal{O}}_{U_x})_x$. 
Up to reparameterizations of $U_x$ by elements of the group $\mrm{Sp}(2n\vert a,b)$, 
the completed ring $(\widehat{\mathcal{O}}_{U_x})_x$ is the stalk of the infinite jet bundle $J^\infty(\underline{\bb{k}}_\bb{M})$.

The assignment $\left((x,\phi),\hat{f}\right)\mapsto \hat{f}_\phi$ therefore determines a map of bundles 
\[\Fr_\bb{M}\times_{\mrm{Sp}(2n\vert a,b)}\widehat{\mathcal{O}}_{2n\vert a,b}\rta J^\infty(\underline{\bb{k}}_\bb{M}).\] 
One constructs an inverse to this map by sending a germ of a function $\hat{g}$ at $x\in\bb{M}$ to a neighborhood $V_x$ on which $\hat{g}$ is defined. 

 \end{proof}

\begin{defn}
Let $\mathcal{V}$ be a symmetric monoidal category. 
Let $A$ be a pro-object in~$\msf{Alg}(\mathcal{V})$. 
An \emph{$A$-module} is an object $N\in\msf{Pro}(\mathcal{V})$ 
together with a map $A\otimes N\rta N$ of pro-objects. 
A \emph{morphism of $A$-modules} is a morphism of pro-objects respecting the action map. 
We let $\msf{Mod}_A(\msf{Pro}(\mathcal{V}))$ denote the category of $A$-modules in $\mathcal{V}$. 
\end{defn}

One can define free, and finitely-generated modules over a pro-object in $\msf{Alg}(\mathcal{V})$ as in the ordinary case. 
Note that the underlying object of $A$ is in $\msf{Pro}(\mathcal{V})$. 

\begin{cor}
Let $A\in\msf{Pro}\left(\msf{Alg}(\msf{Mod}^\mrm{fin}_{(\mfrk{g},K)})\right)$. 
The descent functors induce symmetric monoidal functors

\[\msf{Mod}_A(\msf{Pro}(\msf{Mod}^\mrm{fin}_{(\mfrk{g},K)}))\rta \msf{Mod}_{\mrm{desc}_{(P,\nu)}(A)}(\msf{Pro}(\msf{VB})^\mrm{flat})\]

and

\[\msf{Mod}_A(\msf{Pro}(\msf{Mod}^\mrm{fin}_{(\mfrk{g},K)}))\rta \msf{Mod}_{\mathbf{desc}_{(P,\nu)}(A)}(\msf{Pro}(\msf{Mod}_{\Omega^\bullet_\bb{M}})).\]

\end{cor}

We can forget down 
\[ \msf{Mod}_{\mrm{desc}_{(P,\nu)}(A)}(\msf{Pro}(\msf{Mod}_{\Omega^\bullet_\bb{M}}))\rta \msf{Pro}(\msf{Mod}_{\Omega^\bullet_\bb{M}}).\]
However, the resulting functor 
\[\msf{Mod}_A(\msf{Pro}(\msf{Mod}^\mrm{fin}_{(\mfrk{g},K)}))\rta\msf{Pro}(\msf{Mod}_{\Omega^\bullet_\bb{M}})\]
is only lax-symmetric monoidal. 
The reader should compare this with \cite[Lem. 2.18 and 2.19]{GGW}.

\begin{ex}\label{ex-AisO}
Take $(\mfrk{g},K)$ to be the sHC pair $\gK$ 
and $A$ to be $\widehat{\mathcal{O}}_{2n\vert a,b}$. 
Then we have lax-monoidal functors 

\[\msf{Mod}_{\widehat{\mathcal{O}}_{2n\vert a,b}}(\msf{Mod}_{\gK})\rta\msf{Pro}(\msf{VB}_{/\bb{M}})^\mrm{flat}\]

and

\[\msf{Mod}_{\widehat{\mathcal{O}}_{2n\vert a,b}}(\msf{Mod}_{\gK})\rta\msf{Mod}_{\Omega^\bullet_\bb{M}}.\]

\end{ex}

\begin{defn}\label{def-sGKDescent}
The \emph{super-Gelfand-Kazhdan descent functors} are the functors obtained from Example \ref{ex-AisO} by varying $(P,\nu)$ over $\msf{sGK}^{=}_{2n\vert a,b}$, 

\[\mrm{desc}^\mrm{sGK}\colon \msf{sGK}^=_{2n\vert a,b}\times \msf{Mod}_{\widehat{\mathcal{O}}_{2n\vert a,b}}(\msf{Mod}_{\gK})\rta\msf{Pro}(\msf{VB})^\mrm{flat}\]

and

\[\mathbf{desc}^\mrm{sGK}\colon(\msf{sGK}^=_{2n\vert a,b})^\mrm{op}\times \msf{Mod}_{\widehat{\mathcal{O}}_{2n\vert a,b}}(\msf{Mod}_{\gK})\rta\msf{Mod}_{\Omega^\bullet_\bb{M}}.\]

For $((\bb{M},\omega),\sigma)\in \msf{sGK}_{2n\vert a,b}$, let $\mathbf{desc}_{\bb{M},\sigma}$ denote the resulting functor between module categories.
\end{defn}

\section{Deformation Quantization Descends}\label{sec-DefQuant}

We would like to produce a deformation quantization for symplectic supermanifolds using super-Gelfand-Kazhdan descent. 
In this section, we explain what we mean by deformation quantization, 
and then show how the functor 
$\mathbf{desc}^{\mrm{sGK}}$ of Definition \ref{def-sGKDescent} 
interacts with this process. 

\begin{defn}
Let $A$ be a supercommutative $\bb{k}$-superalgebra. 
A \emph{deformation} of $A$ is an associative $\bb{k}[[\hbar]]$-superalgebra $A_\hbar$ 
together with an isomorphism $A_\hbar/\hbar\simeq A$.
\end{defn}

The commutative algebra we would like to deform is $\mathcal{O}_\bb{M}$ for $\bb{M}$ a symplectic supermanifold. 
By Lemma \ref{lem-superPoisson}, $\mathcal{O}_\bb{M}$ has a Poisson superalgebra structure. 
We would like to consider deformations of $\mathcal{O}_\bb{M}$ that take into account this structure; 
that is, deformations of $\mathcal{O}_\bb{M}$ as a Poisson superalgebra. 
Historically this is done by asking for a deformation $A_\hbar$ of $\mathcal{O}_M$ whose associative product looks like 
\[f\star g=fg+\hbar B_1(f,g)+\hbar^2 B_2(f,g)+\cdots\]
where the $B_i(-,-)$ are bilinear differential operators. 
Since the descent functor $\mathbf{desc}^\mrm{sGK}$ lands in modules over a dg algebra, 
we would like a way to consider Poisson superalgebra in the differential graded setting. 
To do this, and to study deformations quantizations of Poisson dg superalgebras, 
we will use the language of operads. 

\begin{rmk}
We describe a rather general version of deformation of Poisson dg superalgebras below. 
We will only use the special case of $k=1$ to prove our main result Theorem \ref{thm-1}. 
The shifted cases when $k\neq 1$ are of interest for field theories over manifolds of dimension $k\neq 1$. 
The interaction between super-Gelfand-Kazhdan descent and deformation quantization holds in this larger generality, see Lemma \ref{lem-DefQuantDescends}.
\end{rmk}

The following is the super-version of \cite[Def. 2.2.1]{CG1} which can also be found in \cite[Def. 1.1]{CFL}.

\begin{defn}
A \emph{ $\mathcal{P}_k$-algebra in $\msf{Ch}_\bb{k}$} is a cochain complex $A$ of super vector spaces with 

\begin{itemize}

\item a supercommutative product $A\otimes A\rta A$ of degree 0 and 

\item a Lie bracket 
\[\{-,-\}\colon A[k-1]\otimes A[k-1]\rta A[k-1]\]
so that, for every $a\in A$, the map $\{a,-\}$ is a graded superderivation.

\end{itemize}

\end{defn}

See \cite[Def. 2.9]{Sinha} for a construction of the operad $\mathcal{P}_k$ in terms of trees.

\begin{ex}
When $k=1$, a $\mathcal{P}_k$-algebra in $\msf{Ch}_\bb{k}$ is what one might call a  Poisson dg algebra. 
In particular, there is no shift in the bracket.
\end{ex}

\begin{rmk}
For $k\geq 2$, there is an equivalence of operads 
$\mathcal{P}_k\simeq H_\bullet(\mathcal{E}_k)$, 
between the $k$-shifted Poisson operad and the homology of the little $k$-disks operad. 
By formality of the operad $\mathcal{E}_k$ \cite{Cohen}, we have that 
$\mathcal{P}_k$-algebras in chain complexes over a field of characteristic zero are equivalent to algebras over the little $k$-disks operad $\mathcal{E}_k$. 
See for example \cite[Thm. 4.9]{Sinha}. 
\end{rmk}

Next we describe the type of structure a deformation quantization of a $\mathcal{P}_k$-algebra should have. 
The following is \cite[Def. 5.3]{ValerioSafronov}.

\begin{defn}
A \emph{$\mathcal{BD}_1$-algebra in $\msf{Ch}_{\bb{k}[[\hbar]]}$} is a cochain complex $R$ with 

\begin{itemize}

\item an associative multiplication on $R$, and

\item a Lie bracket on $R$, 
\[\{-,-\}\colon R\otimes R\rta R\]
so that, for every $a\in A$, the map $\{a,-\}$ is a graded superderivation, and 
\begin{align}\label{eq-supercommutator}
\hbar\{x,y\}=[x,y]
\end{align}
where $[x,y]$ is the graded supercommutator.
\end{itemize}

\end{defn}

The structure of a $\mathcal{BD}_1$-algebra on a cochain complex $R$ induces a $\mathcal{P}_1$-algebra structure on $R/\hbar$. 
This follows from Equation (\ref{eq-supercommutator}).
One can use this to define an equivalence of operads 
$\mathcal{BD}_1/\hbar\simeq\mathcal{P}_1$.

\begin{rmk}
More generally, one can define an operad $\mathcal{BD}_k$ for $k\geq 2$ 
to be the graded operad obtained from the Rees construction with respect to the Postnikov filtration on $\mathcal{E}_k$, see \cite[\S 5.1]{ValerioSafronov}. 
One then has an equivalence of operads 
\[\mathcal{BD}_k/\hbar\simeq  \mathcal{P}_k\]
This follows, for example, from \cite[Thm. 5.5]{ValerioSafronov}.
\end{rmk}

\begin{defn}
Let $A$ be a $\mathcal{P}_k$-algebra in $\msf{Ch}_\bb{k}$. 
A \emph{$\mathcal{BD}_1$-deformation} of $A$ is a $\mathcal{BD}_k$-algebra $A_\hbar$, 
together with an equivalence of $\mathcal{P}_k$-algebras $A_\hbar/\hbar\simeq A$.
\end{defn}

\begin{lem}\label{lem-DefQuantDescends}
Let $F\colon\mathcal{C}\rta\mathcal{D}$ be a lax symmetric monoidal functor between symmetric monoidal categories tensored over $\bb{k}[[\hbar]]$. 
Let $\mathcal{C}_{\hbar=0}$ and $\mathcal{D}_{\hbar=0}$ denote the corresponding categories tensored over $\bb{k}$. 
Then $F$ induces functors on algebra categories commuting with the quotient map $\bb{k}[[\hbar]]\rta\bb{k}$, 

\[\begin{xymatrix}
{
\msf{Alg}_{\mathcal{BD}_k}(\mathcal{C})\arw[r]^-F\arw[d]^{\hbar=0} & \msf{Alg}_{\mathcal{BD}_k}(\mathcal{D})\arw[d]^{\hbar=0}\\
\msf{Alg}_{\mathcal{P}_k}(\mathcal{C}_{\hbar=0})\arw[r]_-F & \msf{Alg}_{\mathcal{P}_k}(\mathcal{D}_{\hbar=0}).
}
\end{xymatrix}\]
\end{lem}

\begin{proof}
Lax symmetric monoidal functors induce maps on algebra objects, 
given, for example, by 
\[F(R)\otimes F(R)\rta F(R\otimes R)\xrta{F(m)} F(R)\]
where the first arrow is the lax monoidal structure, 
and $m\colon R\otimes R\rta R$ is a multiplication. 

More generally, given an operation $R^{\otimes j}\rta R$, 
the lax structure gives a corresponding operation $F(R)^{\otimes j}\rta F(R)$. 
Thus, lax symmetric monoidal functors induce functors between categories of algebras over operads. 
\end{proof}

\subsubsection{Star Products}

When $k=1$ and our Poisson algebra comes to us as functions on a symplectic supermanifold 
$(\mathcal{O}_{\bb{M}},\{-,-\}_\omega)$, 
we would like our deformations to have an additional property involving the smooth structure on $\bb{M}$. 

\begin{defn}
Let $(\bb{M},\omega)$ be a symplectic supermanifold. 
A \emph{deformation quantization} of $\bb{M}$ is 
a $\mathcal{BD}_1$-deformation $A_\hbar$ of $(\mathcal{O}_{\bb{M}},\{-,-\}_\omega)$ with a $\bb{k}[[\hbar]]$-module isomorphism $A_\hbar\simeq A[[\hbar]]$ 
so that the associative product $\star$ on $f,g\in A_\hbar$ is of the form 

\[f\star g=fg+\hbar B_1(f,g)+\hbar^2B_2(f,g)+\cdots\] 

where the $B_i(-,-)$ are bilinear differential operators on $\bb{M}$.
\end{defn}

\noindent Such a product on $A_\hbar\simeq\mathcal{O}_\bb{M}[[\hbar]]$ is called a \emph{star product}. 

Since the super-Gelfand-Kazhdan descent functor starts from information over the \emph{formal} disk, 
which is not a manifold, 
it does not make sense to ask if star products descend. 
The $\mathcal{BD}_1$-deformation we construct locally will have an obvious form that descends to differential operators globally, 
see the proof of Theorem \ref{thm-SuperFedosov}.

\section{Super-Fedosov Quantization}\label{sec-SuperFedosov}

We would like to prove a super-analogue of Fedosov quantization. 
Recall that Fedosov quantization is the production of a canonical deformation quantization $\cal{A}_D(M)$ of $\cal{O}_M$ 
given a symplectic manifold $M$ together with a sympelctic connection $D$. 
In this section, we will show that given a formal exponential 
$\sigma\in\Gamma(\bb{M},\mrm{Exp}(\bb{M}))$, one can construct a canonical deformation $\mathcal{A}_\sigma(\bb{M})$ of $\mathcal{O}_\bb{M}$ using super-Gelfand-Kazhdan descent. 
See Lemma \ref{lem-ExpConnection} for the relation between a formal exponential on $\bb{M}$ and the data of a super-symplectic connection.

In other words, 
for $\sigma$ a formal exponential on $\bb{M}$, we have an associative algebra
$\cal{A}_\sigma(\bb{M})$ with an isomorphism of Poisson algebras 
\[\cal{A}_\sigma(\bb{M})/\hbar \cong\cal{O}_\bb{M}.\]
We will construct $\cal{A}_\sigma(\bb{M})$ locally over the formal disk, 
and then use the descent construction from Definition \ref{def-sGKDescent}. 
By Lemma \ref{lem-DefQuantDescends}, the descent of a $\mathcal{BD}_1$-deformation is a $\mathcal{BD}_1$-deformation of the descended algebra.

\begin{rmk}
By Lemma \ref{lem-FiberContractible}, the space $\Gamma(\bb{M},\mrm{Exp}(\bb{M}))$ is contractible. 
We therefore obtain an essentially unique deformation quantization of $(\bb{M},\omega)$. 
\end{rmk} 

For motivation, we remind the reader of how this works in the non-super case.

\begin{construction}[Local Fedosov Quantization]\label{con-LocalFedosovQuant}

\noindent We would like to deform
\[\widehat{\cal{O}}_{2n}=\bb{R}[[p_1,\dots,p_n,q_1,\dots,q_n]]\]
using the local symplectic manifold $(\bb{R}^{2n},\omega_0)$ 
where $\omega_0$ is
\[\omega_0=\sum_{i=1}^ndp_i\wedge dq_i.\]
In matrix form,
\[\omega_0(\zeta,\zeta')=-\langle \Omega \zeta,\zeta'\rangle=\zeta^T\Omega \zeta'\]
where 
\[\Omega =\begin{bmatrix}
0 & \mrm{Id}_n\\
-\mrm{Id}_n & 0
\end{bmatrix}.\]
The Poisson bracket on $\widehat{\cal{O}}_{2n}$ is 
\[\{f,g\}=-(\nabla f)^T\Omega (\nabla g)=\sum_{i=1}^n\frac{\del f}{\del p_i}\frac{\del g}{\del q_i}-\frac{\del f}{\del q_i}\frac{\del g}{\del p_i}.\]
From the Poisson bracket, we can abstract a bivector
\[\alpha=\sum_{i=1}^n \frac{\del}{\del p_i}\otimes\frac{\del}{\del q_i}-\frac{\del}{\del q_i}\otimes\frac{\del}{\del p_i}.\]
The deformation of $\widehat{\cal{O}}_{2n}$ has underlying vector space

\[\widehat{\cal{A}}_{2n}=\bb{R}[[p_1,\dots,p_n,q_1,\dots,q_n,\hbar]]\]

with product given by
\[f\star g=m\left(\mrm{exp}\left(\frac{\hbar}{2}\alpha\right)(f\otimes g)\right).\]
Here, $m$ is multiplication of power series 
$\widehat{\cal{A}}_{2n}\otimes\widehat{\cal{A}}_{2n}\rta\widehat{\cal{A}}_{2n}$.
On generators, the product is given by
\[p_i\star q_j=\frac{\hbar}{2} \delta_{ij}\]
and 
\[q_i\star p_i=-\frac{\hbar}{2}\delta_{ij}\]
with the rest of the products being zero.
The algebra $\widehat{\cal{A}}_{2n}$ is sometimes called the \emph{Weyl algebra}. 

\end{construction}

\begin{construction}[Local Super-Fedosov Quantization]\label{con-LocalSuperFedosovQuant}

\noindent We would like to replicate the above construction in the super context.
Thus, we want to deform the Poisson superalgebra $\widehat{\cal{O}}_{2n\vert a,b}$, 
whose underlying superalgebra is
\[\widehat{\mrm{Sym}}(p_1,\dots,p_n,q_1,\dots,q_n,\theta_1,\dots,\theta_r)\]
where $|p_i|=|q_i|=0$ and $|\theta_i|=1$. 
Our deformation will be constructed using the local picture of the symplectic supermanifold
$(\bb{R}^{2n\vert r},\omega_Q)$ 
from Example \ref{ex-LocalSymplecticStructure} where $Q$ has signature $(a,b)$. 
Here, $Q$ is a symmetric, nondegenerate bilinear form, with corresponding matrix $(g^{ij})$, 
and $\omega_Q$ is 
\[\omega_Q=\sum_{i=1}^ndp_i\wedge dq_i+\sum_{i,j=1}^r g^{ij} d\theta_i\otimes d\theta_j.\]
In matrix form,
\[\omega_Q(\xi,\xi')=\langle H_Q\xi,\xi'\rangle=-\xi^{sT}H_Q\xi'\]
where 
\[H_Q=\begin{bmatrix}
\Omega  & 0\\
0 & G
\end{bmatrix}\]
and $G=(g^{ij})$. 
The Poisson bracket on $\widehat{\cal{O}}_{2n\vert a,b}$ is given by
\[\{f,g\}=-(\nabla f)^{sT}H_Q(\nabla g),\]
we get an associated bivector

\[\tilde{\alpha}=\sum_{i=1}^n\frac{\del}{\del p_i}\otimes\frac{\del}{\del q_i}-\frac{\del }{\del q_i}\otimes\frac{\del}{\del p_i}-\sum_{i=1}^r g^{ij}\left(\frac{\del}{\del \theta_i}\otimes\frac{\del}{\del \theta_j}\right).\]

\end{construction}

Using the same idea as in the ordinary case, we make the following definition:

\begin{defn}\label{def-AQ}
Let $\AQ$ be the superalgebra with underlying super vector space 

\[\AQ=\widehat{\mrm{Sym}}(p_1,\dots,p_n,q_1,\dots,q_n,\theta_1,\dots,\theta_r,\hbar)\]

where $p_i,q_i,\hbar$ are even and $\theta_i$ are odd,
and with product
\[f\star g=m\left(\mrm{exp}\left(\frac{\hbar}{2}\tilde{\alpha}\right)(f\otimes g)\right).\]
\end{defn} 

On generators, the product is given by
\begin{align*}
p_i\star q_j&=\frac{\hbar}{2} \delta_{ij}\\
q_i\star p_j&=-\frac{\hbar}{2}\delta_{ij}\\
\theta_i\star \theta_j&=\theta_i\theta_j-\frac{\hbar}{2}g^{ij} . 
\end{align*}

\begin{prop}\label{prop-AQisBD}
The superalgebra $\AQ$ is a $\mathcal{BD}_1$-deformation of the Poisson superalgebra~$\widehat{\mathcal{O}}_{2n\vert a,b}$. 
\end{prop}

\begin{proof}
By construction, there is an equivalence of $\bb{k}[[\hbar]]$-modules 

\[\AQ= \widehat{\mathcal{O}}_{2n\vert a,b}[[\hbar]].\] 

Quotienting by $\hbar$, the product on $\AQ$ becomes 
the multiplication $m(f\otimes g)$ on $\widehat{\mathcal{O}}_{2n\vert a,b}$. 
Lastly, the super-commutator bracket on generators is given by

\begin{align*}
[p_i,q_i]&=\frac{\hbar}{2} \delta_{ij}+\frac{\hbar}{2}\delta_{ij}=\hbar\delta_{ij}=\hbar\{p_i,q_i\}\\
[\theta_i,\theta_j]&=\left(\theta_i\theta_j-\frac{\hbar}{2}g^{ij}\right)+\left(\theta_j\theta_i-\frac{\hbar}{2}g^{ij}\right)=\theta_i\theta_j-\theta_i\theta_j-\hbar g^{ij}=-\hbar g^{ij}=\hbar\{\theta_i,\theta_j\}.
\end{align*}
This is the Poisson bracket on $\widehat{\mathcal{O}}_{2n\vert a,b}$ form Example \ref{ex-FormalPoisson}. 

\end{proof}

\subsubsection{Deformation Quantization and Super-Gelfand-Kazhdan Descent}\label{subsec-New}

For the super-Gelfand-Kazhdan descent functor of Definition \ref{def-sGKDescent}, 
we would like to apply Lemma \ref{lem-DefQuantDescends} in the case 

\[\mathcal{C}= \msf{Mod}_{\widehat{\mathcal{O}}_{2n\vert a,b}}(\msf{Mod}_{\gK})(\mrm{Vect}_{\bb{k}[[\hbar]]}),\]

\[\mathcal{D}=\msf{Mod}_{\Omega^\bullet_\bb{M}[[\hbar]]},\]

and $F$ is descent $\mathbf{desc}^\mrm{sGK}_{(\bb{M},\sigma)}$ on the level of $\bb{k}[[\hbar]]$-modules. 
In this case, given an algebra $A\in\mathcal{C}$, we would like to consider deformation quantizations $A_\hbar$ of $A$ that live in $\mathcal{C}$. 
We will only use the deformations $A_\hbar$ that are isomorphic to $A_\hbar[[\hbar]]$ as $\bb{k}[[\hbar]]$-modules. 
           
By Lemma \ref{lem-PoissonSuperalgebra}, there is a Poisson superalgebra structure on formal functions $\widehat{\mathcal{O}}_{2n\vert a,b}$, making it an object  

\[\widehat{\mathcal{O}}_{2n\vert a,b}\in\msf{Alg}_{\mathcal{P}_1}(\msf{Mod}_{\widehat{\mathcal{O}}_{2n\vert a,b}}(\msf{Mod}_{\gK})).\]

Let $(\bb{M},\sigma)\in\msf{sGK}_{2n\vert a,b}$. 
By Lemma \ref{lem-DescO}, we have an equivalence 

\[\Gamma\left(\bb{M},\mathbf{desc}^\mrm{sGK}_{\bb{M},\sigma}\left(\widehat{\mathcal{O}}_{2n\vert a,b}\right)\right)=\mathcal{O}_\bb{M}.\] 

By Construction \ref{con-LocalSuperFedosovQuant}, 
we have a deformation quantization 
$\mathcal{A}_{2n\vert a,b}$ of $\widehat{\mathcal{O}}_{2n\vert a,b}$. 
We would like to apply Lemma \ref{lem-DefQuantDescends} to say that

\[\Gamma\left(\bb{M},\mathbf{desc}^\mrm{sGK}_{\bb{M},\sigma}\left(\mathcal{A}_{2n\vert a,b}\right)\right)\]

is a $\mathcal{BD}_1$-deformation of $\mathcal{O}_\bb{M}$.  
The one hiccup here is that $\AQ$ is not a $\mfrk{g}_{2n\vert a,b}$-module. 
We can fix this by replacing $\mfrk{g}_{2n\vert a,b}$ with a Lie superalgebra involving $\hbar$. 

\begin{notation}
Let $\mfrk{g}^\hbar_{2n\vert a,b}$ be the Lie superalgebra of derivations $\mrm{Der}(\AQ)$ of $\AQ$ as a graded module over the graded algebra $\bb{K}$.
\end{notation}

Now $\mfrk{g}^\hbar_{2n\vert a,b}$ acts on $\AQ$. 
Moreover, since $\AQ/\hbar$ is $\widehat{\mathcal{O}}_{2n\vert a,b}$, 
we get an action by derivations of $\mfrk{g}^\hbar_{2n\vert a,b}$ on $\widehat{\mathcal{O}}_{2n\vert a,b}$, 

\begin{align}\label{eq-map}
\mrm{Aut}^\mrm{grad}_\bb{K}(\AQ)\xrta{(-)\otimes \bb{k}}\mrm{Aut}^\mrm{grad}_\bb{k}(\widehat{\mathcal{O}}_{2n\vert a,b})\xrta{\mrm{forget}}\mrm{Aut}_\bb{k}(\widehat{\mathcal{O}}_{2n\vert a,b}).
\end{align}

This action factors through the action by Poisson derivations.
In other words, since the Lie superalgebra $\mfrk{g}_{2n\vert a,b}$ is given by derivations of $\widehat{\mathcal{O}}_{2n\vert a,b}$ that respect the Poisson structure coming from the symplectic form $\omega_Q$,  
there is a Lie superalgebra map

\[\mfrk{g}^\hbar_{2n\vert a,b}\rta\mfrk{g}_{2n\vert a,b}.\]

\noindent See \cite[\S 3.2]{BK} for similar statements in the purely even case.  
Note that, as in \cite[\S 3.2]{BK}, $\mrm{Aut}^\mrm{grad}_\bb{K}(\AQ)$ is a pro-algebraic super group. 

We would like to apply a variant of super-Gelfand-Kazhdan descent for the sHC pair $\gKh$ instead of $\gK$. 
To do so, we need an analogue of the functor 

\[\mrm{Fr}\colon\msf{sGK}^=_{2n\vert a,b}\rta\msf{Bun}^\mrm{flat}_{\gK}\]

from Corollary \ref{cor-FrFunctor}. 
That is, we need a way of equipping the symplectic frame bundle $\Fr_\bb{M}$ with the structure of a $\gKh$-bundle. 
This is done by replacing the principal 
$\mrm{Aut}_{2n\vert a,b}$-bundle $\bb{M}^\mrm{coor}\rta\bb{M}$ 
with the principal $\mrm{Aut}^\mrm{filt}(\AQ)$-bundle $\bb{M}^\mrm{coor}_\hbar$ 
whose fiber over a point $x\in\bb{M}$ is 

\[\mrm{Isom}^\mrm{grad}_\bb{K}(\widehat{\mathcal{O}}_{\bb{M},x}[[\hbar]],\AQ).\]

\noindent See also \cite[Pg. 18]{FFSh} in the purely even case. 

We obtain a map $\bb{M}^\mrm{coor}_\hbar\rta\bb{M}^\mrm{coor}$ over $\bb{M}$ given by the map (\ref{eq-map}) fiberwise. 

Just as in Lemma \ref{lem-McoorConn}, 
$\bb{M}^\mrm{coor}_\hbar$ has a flat connection $\nu_\hbar^\mrm{coor}$, which now takes values in 
\[\mrm{Lie}(\mrm{Aut}^\mrm{grad}_\bb{K}(\AQ))=\mfrk{g}^\hbar_{2n\vert a,b}.\]

As before, we use a type of formal exponential to pullback this connection to a connection on $\Fr_\bb{M}$. 

\begin{defn}\label{def-hExp}
Let $(\bb{M},\omega)$ be a symplectic supermanifold of type $(2n\vert a,b)$. 
An \emph{$\hbar$-formal exponential} on $\bb{M}$ is a section of the bundle 
\[\mrm{Exp}^\hbar(\bb{M})=\bb{M}^\mrm{coor}_\hbar/\mrm{Sp}(2n\vert a,b).\]
\end{defn}

\noindent See also \cite[Pg. 18]{FFSh}. 

\begin{variant}\label{var-sgk=h}
Let $\msf{sGK}^{=,\hbar}_{2n\vert a,b}$ be as in Variant \ref{var-sgk=} but with $\hbar$-formal exponentials.
\end{variant}

In the super case, we have the following analogue of Lemma \ref{lem-FiberContractible}. 

\begin{lem}\label{lem-FiberContractible2}
The space $\mrm{Aut}^\mrm{grad}_\bb{K}(\AQ)/\mrm{Sp}(2n\vert a,b)$ is contractible, 
and thus $\hbar$-formal exponentials always exist.
\end{lem}

\begin{proof}
As in \cite[Pg. 24]{BK}, we have a short exact sequence 

\[1\rta\mrm{ker}(P)\rta\mrm{Aut}^\mrm{grad}_\bb{K}(\AQ)\xrta{P} \mrm{Sp}(2n\vert a,b)\rta 1\]

and $\mrm{ker}(P)$ is pro-unipotent, hence pro-nilpotent and contractible. 
As in Lemma \ref{lem-FiberContractible}, $\mrm{ker}(P)$ is a pro-vector space.
\end{proof}

As in Proposition \ref{prop-gkfromExp}, given an $\hbar$-formal exponential $\sigma_\hbar$ on $\bb{M}$, 
we get a $\gKh$-bundle structure on $\Fr_\bb{M}$. 

\begin{lem}\label{lem-hExptoExp}
An $\hbar$-formal exponential $\sigma_\hbar$ on $\bb{M}$ induces a formal exponential $\sigma$ on $\bb{M}$. 
Moreover, the induced connection 1-form $\tilde{\sigma}_\hbar^*(\nu_\hbar^\mrm{coor}))$ on $\Fr_\bb{M}$ hits the induced connection 1-form  
$\tilde{\sigma}^*(\nu^\mrm{coor}))$ 
under the map 
\[\Omega^1(\Fr_\bb{M};\mfrk{g}^\hbar_{2n\vert a,b})\rta\Omega^1(\Fr_\bb{M};\mfrk{g}_{2n\vert a,b}).\]
\end{lem}

\begin{proof}
The map $\bb{M}^\mrm{coor}_\hbar\rta\bb{M}^\mrm{coor}$ induces a map 
$\mrm{Exp}^\hbar(\bb{M})\rta\mrm{Exp}(\bb{M})$. 
Composing with this map takes an $\hbar$-formal exponential to a formal exponential. 
The second claim follows from the fact that these bundle maps and the map 
$\mfrk{g}^\hbar_{2n\vert a,b}\rta\mfrk{g}_{2n\vert a,b}$ 
are both defined by the map (\ref{eq-map}).
\end{proof}

We can now consider super-Gelfand-Kazhdan descent for $\gKh$-modules, 

\[\mathbf{desc}^{\mrm{sGK}}_{(\bb{M},\sigma_\hbar)}\colon\msf{Mod}_{\gKh}\rta\msf{Mod}_{\Omega^\bullet_\bb{M}[[\hbar]]}.\]

Using the same notation for this variant is somewhat justified by the following lemma. 

\begin{lem}
Let $\sigma_\hbar$ be an $\hbar$-formal exponential on $\bb{M}$ with induced formal exponential $\sigma$. 
Then there is a commutative diagram 

\[\begin{xymatrix}
{
\msf{Mod}_{\gKh}\arw[rr]^-{\mathbf{desc}^{\mrm{sGK}}_{(\bb{M},\sigma_\hbar)}} & &\msf{Mod}_{\Omega^\bullet_\bb{M}[[\hbar]]}\\
\msf{Mod}_{\widehat{\mathcal{O}}_{2n\vert a,b}}(\msf{Mod}_{\gK})\arw[rr]_-{\mathbf{desc}^{\mrm{sGK}}_{(\bb{M},\sigma)}}\arw[u]^r & &\msf{Mod}_{\Omega^\bullet_\bb{M}}\arw[u]^s
}
\end{xymatrix}\]

where the left vertical arrow is given by restricting the module structure along $\mfrk{g}^\hbar_{2n\vert a,b}\rta\mfrk{g}_{2n\vert a,b}$ 
and the right vertical arrow is given by restriction along the map setting $\hbar=0$.
\end{lem}

\begin{proof}
Let $V$ be in $\msf{Mod}_{\widehat{\mathcal{O}}_{2n\vert a,b}}(\msf{Mod}_{\gK})$. 
Then both $\mathbf{desc}_{(\bb{M},\sigma)}^\mrm{sGK}(V)$ and 
$\mathbf{desc}_{(\bb{M},\sigma_\hbar)}^\mrm{sGK}(r(V))$ 
are given by taking horizontal forms of the vector bundle 
$\Fr_\bb{M}\times_{\mrm{Sp}(2n\vert a,b)}V$, 
but with respect to possibly different flat connections. 
By the discussion around Equation (\ref{eq-Rho}), 
in the first case, the differential induced by the flat connection is 
\[\rho_\mfrk{g}^V(\tilde{\sigma}^*(\nu^\mrm{coor}))+d_\mrm{dR}\] 
and in the latter case by 
\[\rho_{\mfrk{g}^\hbar_{2n\vert a,b}}^{r(V)}(\tilde{\sigma}_\hbar^*(\nu_\hbar^\mrm{coor}))+d_\mrm{dR}.\] 

By Lemma \ref{lem-hExptoExp}, the action of $\tilde{\sigma}_\hbar^*(\nu_\hbar^\mrm{coor}))$ on $r(V)$ is the same as the action of $\tilde{\sigma}^*(\nu^\mrm{coor}))$ on $V$. 

\noindent Thus, the flat connections are the same and therefore the diagram commutes.
\end{proof}

One consequence of this result is that Lemma \ref{lem-DescO} still holds when viewing $\widehat{\mathcal{O}}_{2n\vert a,b}$ as a $\gKh$-module. 
Now $\AQ$ is a deformation of $\widehat{\mathcal{O}}_{2n\vert a,b}$ in $\gKh$-modules, 
and by Lemma \ref{lem-DefQuantDescends}
should descend to a deformation of $\mathcal{O}_\bb{M}$.  

In summary, we have the following,

\begin{cor}
Given a $\hbar$-formal exponential $\sigma_\hbar$ on $\bb{M}$ super-Gelfand-Kazhdan descent for $\gKh$-modules 
takes a $\mathcal{BD}_1$-deformation $\AQ$ of $\widehat{\mathcal{O}}_{2n\vert a,b}$
 to a $\mathcal{BD}_1$-deformation of $\mathcal{O}_\bb{M}$. 
\end{cor}

\subsubsection{Main Theorem}

\begin{thm}\label{thm-SuperFedosov}
Let $(\bb{M},\omega)$ be a symplectic supermanifold. 
The assignment 
\[\sigma\mapsto\Gamma(\bb{M},\mathbf{desc}^\mrm{sGK}_{(\bb{M},\sigma)}(\AQ))\]
defines a map
\[\mathcal{A}_{(-)}(\bb{M})\colon \Gamma(\bb{M},\mrm{Exp}^\hbar(\bb{M}))\rta Q(\bb{M},\omega)\]
from the set of $\hbar$-formal exponentials of $\bb{M}$ to the set of equivalence classes of deformation quantizations of $(\bb{M},\omega)$. 
\end{thm}

\begin{proof}
Let $(\bb{M},\sigma)$ be a symplectic supermanifold and $\hbar$-formal exponential. 
By Lemma \ref{lem-DescO}, the degree zero piece of 
$\mathbf{desc}^\mrm{sGK}_{(\bb{M},\sigma)}(\widehat{\cal{O}}_{2n\vert a,b})$ 
is $\cal{O}_{\bb{M}}$. 
By Lemma \ref{lem-DefQuantDescends}, $\mathcal{O}_{\bb{M}}$ is a Poisson superalgebra. 
Let $\mathcal{A}_\sigma(\bb{M})$ be the degree zero piece of 
$\mathbf{desc}^\mrm{sGK}_{(\bb{M},\sigma)}(\AQ)$. 
By Lemma \ref{lem-DefQuantDescends}, 
$\mathcal{A}_\sigma(\bb{M})$ is a $\mathcal{BD}_1$-deformation of the Poisson superalgebra $\mathcal{O}_{\bb{M}}$. 

It remains to check that the product on $\mathcal{A}_\sigma(\bb{M})$ is a star product. 
One can check this locally, where $\bb{M}$ looks like $\bb{R}^{2n\vert a+b}$. 
Here, the $\hbar$ terms in the product on $\mathcal{A}_\sigma(\bb{M})$ are given in terms of the partial derivatives $\frac{\del}{\del p_i},\frac{\del}{\del q_i},\frac{\del}{\del \theta_i}$, which are differential operators. 
\end{proof}

\begin{ex}\label{rmk-LiftDefo}
By Remark \ref{rmk-VBsGK}, 
we have a functor 
\[\tilde{L}\colon\msf{VB}^{\mrm{quad},\nabla}_{/X}\rta\msf{sGK},\]
so that every symplectic supermanifold of the form $(\pi^*E)[1]$ from Example \ref{ex-Rothstein} has a natural choice of formal exponential $\sigma_E$. 
We can upgrade this choice to an $\hbar$-formal exponential. 

\begin{lem}\label{lem-Connhbar}
If $\bb{M}=E[1]$ is the symplectic supermanifold defined in Example \ref{ex-Rothstein} from the data $(M,\omega,E,g,\nabla)$, 
then a symplectic connection on $M$ determines an $\hbar$-formal exponential on $\bb{M}$.
\end{lem}

\begin{proof}
The argument is the same as in the proof of Lemma \ref{lem-ExpConnection} after noting that a symplectic connection on $M$ defines not just a compatible choice of isomorphisms of Poisson algebras $\widehat{\mathcal{O}}_{2n}\simeq\widehat{\mathcal{O}}_{M,x}$, 
but isomorphisms of $\bb{K}$-modules 
\[\mrm{Weyl}_{2n}\simeq\widehat{\mathcal{O}}_{M,x}[[\hbar]].\]
Similarly, the metric connection $\nabla$ on $E$ induces a compatible family of isomorphisms
\[\AQ\simeq\widehat{\mathcal{O}}_{\bb{M},x}.\]
of an $\hbar$-formal exponential.
\end{proof}

The assignment $E\mapsto \mathcal{A}_{\sigma_E}((\pi^*E)[1])$ defines a functor 
\[\tilde{A}_X\colon\msf{VB}^{\mrm{quad},\nabla}_{/X}\rta\msf{sAlg}(\msf{Ch}_{\bb{k}[[\hbar]]}),\]
as discussed in \S\ref{subsec-invariants}.
\end{ex}

\begin{rmk}\label{rmk-BKours3}
If, following \cite[Lem. 3.4]{BK}, one wishes to construct all deformation quantizations, 
one would apply the same process but using descent for a much larger sHC pair. 
The descent functor we use (from Definition \ref{def-sGKDescent}) uses the sHC pair $\gKh$. 
As explained in Remarks \ref{rmk-BKours} and \ref{rmk-BKours2}, 
this corresponds to only allowing our gluing data to come from \emph{linear} maps $\mrm{Sp}(2n\vert a,b)$. 
To surject onto $Q(\bb{M},\omega)$, 
one would like to consider lifts of $\bb{M}^\mrm{coor}$ from a 
$(\mfrk{g}_{2n\vert a,b},\mrm{Aut}_{2n\vert a,b})$-bundle 
to a $(\mfrk{g}^\hbar_{2n\vert a,b},\mrm{Aut}(\AQ))$-bundle 
instead of choosing $\bb{M}^\mrm{coor}_\hbar$ and pulling back along an $\hbar$-formal exponential. 
Here, $(\mfrk{g}^\hbar_{2n\vert a,b},\mrm{Aut}(\AQ)$ is a super-version of the HC pair $(\mrm{Der}(D),\mrm{Aut}(D))$ in \cite[Lem. 3.4]{BK}; 
that is, derivations and automorphism of the algebra $\AQ$.
\end{rmk}

\subsection{Description in terms of Weyl and Clifford Algebras}\label{subsec-WeylClifford}

We review the basic definitions of Weyl and Clifford algebras as a means of establishing notation. 
Then, we give a description of the super-Fedosov quantization in terms of these algebras.

\begin{defn}
Fix $n$. The \emph{Heisenberg Lie superalgebra} over $\bb{k} [[\hbar]]$, 
denoted $\mfrk{h}_{2n}$, 
is the Lie superalgebra with even generators
$p_1,\dots,p_n,q_1,\dots,q_n$
and Lie brackets given by $[p_i,q_i]=\hbar$, and zero otherwise.
\end{defn}

\begin{defn}\label{def-WeylAlgebra}
The \emph{Weyl algebra} of a symplectic vector space $(V,\omega)$ is the quotient

\[\mrm{Weyl}(V,\omega)\colon= T(V[\hbar,\hbar^{-1}])/I_\omega,\]

where $T(V)$, the tensor algebra, is taken over $\bb{k} [\hbar,\hbar^{-1}]$, 
and $I_\omega$ is the ideal generated by the set 

\[\{u\otimes v-v\otimes u-\omega(u,v)\hbar:u,v\in V\}.\]

\end{defn}

\begin{ex}
With notation as in Example \ref{ex-LocalSymplecticStructure},
the Weyl algebra of the symplectic vector space $(T^*_0\bb{R}^n,\omega_0)$ 
is the enveloping algebra of the Heisenberg Lie algebra, 

\[\mrm{Weyl}(T^*_0\bb{R}^n,\omega_0)=U(\mfrk{h}_{2n}).\] 

\end{ex}

\begin{defn}
For fixed $a,b$, the \emph{Clifford Lie superalgebra} over $\bb{k}[[\hbar]]$, 
denoted $\mfrk{cl}_{a,b}$, 
is the Lie algebra with odd generators
$\gamma_1,\dots,\gamma_a,\psi_1,\dots,\psi_b$
and brackets zero except 
\[[\gamma_i,\gamma_i]=\hbar\]
 and 
 \[[\psi_i,\psi_i]=-\hbar.\] 
\end{defn}

\begin{defn}
The \emph{Clifford algebra} of a super vector space equipped with a quadratic function $(V,Q)$ is the quotient

\[\mrm{Cliff}(V,Q)\colon= T(V[\hbar,\hbar^{-1}])/I_Q\]

where $I_Q$ is the ideal generated by the set 

\[\{v\otimes v-Q(v)\hbar:v\in V\}.\]

\end{defn}

\begin{ex}
With notation as in Example \ref{ex-LocalSymplecticStructure}, 
let $a+b=r$ and $Q$ be a quadratic function on $\bb{R}^r$ with signature $(a,b)$. 
Then the Clifford algebra of the symplectic super vector space $(\bb{R}^{0\vert r},\omega_Q)$ 
is the enveloping algebra of the Clifford Lie algebra, 

\[\mrm{Cliff}(\bb{R}^{0\vert a,b},Q)=U(\mfrk{cl}_{a,b}).\]

\end{ex}

The deformation of $\widehat{\mathcal{O}}_{2n\vert a,b}$ that we are interested in is a mixture of Weyl and Clifford algebras. 
Since we are working with formal functions, 
we are interested in power series rings. 
Replacing the tensor algebra with the completed tensor algebra in the definitions of 
the Weyl, Clifford, and enveloping algebras, 
we obtain notations of 
a completed Weyl algebra $\widehat{\mrm{Weyl}}$, 
a completed Weyl algebra $\widehat{\mrm{Cliff}}$, 
and a completed enveloping algebra $\widehat{U}$. 
Since the generators of the Clifford algebra, which are odd, square to zero, 
completing does not change the algebra.

\begin{lem}\label{lem-AQasWeylCliff}
Let $Q$ be a quadratic function on $\bb{R}^r$ with signature $(a,b)$. 
There is an equivalence of superalgebras 

\[\widehat{\mrm{Weyl}}(\bb{R}^{2n\vert 0},\omega_0)\otimes\mrm{Cliff}(\bb{R}^{0\vert r},\omega_Q)\xrta{\sim}\AQ.\]

\end{lem}

One can make further (notational) identifications,

\[\AQ\cong \widehat{U}(\mfrk{h}_{2n})\otimes_{\bb{k}[[\hbar]]} U(\mfrk{cl}_{a,b}).\]

\begin{proof}
The underlying super vector spaces are both 

\[\widehat{\mrm{Sym}}(p_1,\dots,p_n,q_1,\dots,q_n,\theta_1,\dots,\theta_r,\hbar).\]

By the proof of Proposition \ref{prop-AQisBD}, 
the super-commutator bracket in $\AQ$ agrees with the Lie bracket of 
$\mfrk{h}_{2n}$ and $\mfrk{cl}_{a,b}$. 
By the universal property of enveloping algebras, 
we obtain a map of algebras 

\[\widehat{U}(\mfrk{h}_{2n})\otimes_{\bb{k}[[\hbar]]}U( \mfrk{cl}_{a,b})\rta\AQ\] 

which is an isomorphism on underlying super vector spaces, and hence an isomorphism of algebras.
\end{proof}

One should compare this to \cite[\S 1.4]{Engeli}.

\bibliography{vabib}
\bibliographystyle{alpha}

\end{document}